\documentclass[11pt, reqno]{amsart}

\usepackage{amssymb, amscd, latexsym}
\usepackage{enumitem}
\usepackage{amsmath}
\usepackage{amsthm}
\usepackage{bm}
\usepackage{graphicx}
\usepackage[all]{xy}
\usepackage{float}
\usepackage{mathrsfs}
\usepackage{pgf,tikz}
\usetikzlibrary{arrows}
\usepackage{array}
\usepackage{arydshln}
\usepackage{hyperref}
\hypersetup{
pdftitle={Research Paper},
pdfauthor={Ali Akbar Yazdan Pour},
pdfsubject={Betti Number and chordality},
pdfcreator={Ali Akbar Yazdan Pour},
colorlinks,
linkcolor=blue,
citecolor=magenta,
anchorcolor=red,
bookmarksopen,
urlcolor=red,
filecolor=red,
pdfpagetransition={Wipe}
}

\theoremstyle{plain}
\newtheorem{thm}{Theorem}[section]
\newtheorem{cor}[thm]{Corollary}
\newtheorem{lem}[thm]{Lemma}
\newtheorem{prop}[thm]{Proposition}

\theoremstyle{definition}
\newtheorem{defn}[thm]{Definition}
\newtheorem{ex}[thm]{Example}

\theoremstyle{remark}
\newtheorem{rem}{Remark}

\def\cocoa{{\hbox{\rm C\kern-.13em o\kern-.07em C\kern-.13em o\kern-.15em A}}}

\def\reg{{\rm reg}}

\def\implies{\ifmmode\Rightarrow \else
        \unskip${}\Rightarrow{}$\ignorespaces\fi}
\def\Tor{\mathrm{Tor}}

\def\xb{{\mathbf{x}}}
\def\eb{{\mathbf{e}}}
\def\C{\mathcal{C}}
\def\supp{\mathrm{supp}}
\def\D{\mathcal{D}}
\begin{document}
	
\title[Multigraded resolution]{Multigraded minimal free resolutions of simplicial subclutters}
\author[M. Bigdeli]{Mina Bigdeli}
\address{Mina Bigdeli\\ School of Mathematics, Institute for Research in Fundamental Sciences (IPM),  P.O.Box: 19395-5746, Tehran, Iran}
\email{mina.bigdeli98@gmail.com, mina.bigdeli@ipm.ir}

\author[A. A. Yazdan Pour]{{Ali Akbar} {Yazdan Pour}}
\address{Ali Akbar Yazdan Pour\\ Department of Mathematics Institute for Advanced Studies in Basic Sciences (IASBS)\\ P.O.Box 45195-1159, Zanjan, Iran}
\email{yazdan@iasbs.ac.ir}
	
	
\begin{abstract}
This paper concerns the study of a class of clutters called simplicial subclutters. Given a clutter $\mathcal{C}$ and its simplicial subclutter $\mathcal{D}$, we compare some algebraic properties and invariants of the ideals $I, J$ associated to these two clutters,  respectively. We give a formula for computing the (multi)graded Betti numbers of $J$ in terms of those of $I$ and some combinatorial data about $\mathcal{D}$. As a result, we see that if $\mathcal{C}$ admits a simplicial subclutter,  then there exists a monomial $u \notin I$ such that the (multi)graded Betti numbers of $I+(u)$ can be computed through those of $I$. It is proved that the Betti sequence of any graded ideal with linear resolution is the Betti sequence of an ideal associated to a simplicial subclutter of the complete clutter. These ideals turn out to have linear quotients. However, they do not form all the equigenerated square-free monomial ideals with linear quotients. If $\mathcal{C}$ admits $\varnothing$ as a simplicial subclutter, then $I$ has linear resolution over all fields. Examples show that the converse is not true.
\end{abstract}
\thanks{Mina Bigdeli was supported by a grant from IPM}
\maketitle

\section*{introduction}

Free resolutions have been a  central topic in commutative algebra since the work of David Hilbert and his celebrated theorem ``Hilbert Syzygy Theorem". They are a very important tool to study the properties of a graded module over finitely generated graded $\mathbb{K}$-algebras. They, in particular, are used to compute the Castelnuovo--Mumford regularity, projective dimension, depth, Hilbert function, etc. It is, in general very difficult to determine the whole resolution of a module, even using some computer programming systems.  However, algebraists try to find some tools to investigate the general shape of the resolution and to compute some numerical data explaining them. An important  invariant, regarding the graded minimal free resolutions, which carries most of the numerical data about them, is the (multi)graded Betti numbers. For $\textbf{\rm \textbf{a}} \in \mathbb{Z}^n$ and $i\in \mathbb{Z}$, the $(i,\textbf{\rm \textbf{a}})$-th Betti number of a finitely generated multigraded module $M$ over a multigraded standard $\mathbb{K}$-algebra $R$ is given by
$$\beta_{i, \textbf{\rm \textbf{a}}}(M)=\dim_{\mathbb{K}}\Tor^R_i(\mathbb{K};M)_{\textbf{\rm \textbf{a}}}.$$
While it is again too hard to determine the exact value of the (multi)graded Betti numbers, one may try to explain them in combinatorial terms for monomial ideals over the polynomial rings. Indeed, given a monomial ideal $I$, passing through polarization, one can obtain a square-free monomial ideal $J$ whereas lots of algebraic properties are preserved via this operation. Among all other properties, $I$ and $J$ share same graded Betti numbers. Hence, in order to study the graded free resolution of monomial ideals,  it is fair enough to concentrate on the square-free ones. On the other hand, there is a bijection between the class of square-free monomial ideals over a polynomial ring and some classes of combinatorial or topological objects. Given a square-free monomial ideal $I$ over the polynomial ring $S=\mathbb{K}[x_1,\ldots,x_n]$ one can associate a simplicial complex $\Delta_I$ to $I$ such that $F\subset [n]=\{1,\ldots,n\}$ does not belong to $\Delta_I$ if and only if $\mathbf{x}_F:=\prod_{i\in F}x_i$ belongs to $I$. A more combinatorial approach associates a collection $\mathcal{C}_I$ of subsets of $[n]$, called a clutter, to  the minimal generating set of $I$ as follows:  A subset $F\subset [n]$ belongs to $\C_I$ if and only if $\mathbf{x}_F$ is an element in the minimal generating set of $I$. In case $I$ is equigenerated, say in degree $d$, the complement of $\C_I$, denoted by $\bar{\C}_I$, is defined to be the collection of all $d$-subsets of $[n]$ not belonging to $\C_I$. Unless otherwise stated, throughout the paper, by the ideal attached to a clutter $\C$, we mean the ideal whose generators are in correspondence with the elements of the complement of $\C$.

 \medskip
 In this paper, we study a class of clutters, called simplicial subclutters.  In order to define this class, we first introduce, in Section~\ref{Main}, a reduction operation on  the uniform clutters -- the clutters whose elements have the same cardinality. This operation on a clutter $\C$ is based on removing an element $F$ which contains a particular subset, called simplicial element of $\C$. The definition of simplicial element and all other preliminaries is given in Section~\ref{prelim}.   A clutter has a simplicial subclutter if and only if it admits a non-trivial simplicial element. Not all clutters have such elements, see e.g Example~\ref{duncehat example}.   In Section~\ref{Main}, we deal with the ideals associated to  uniform clutters and their simplicial subclutters. Given a uniform clutter $\C$ and its simplicial subclutter $\D$, we investigate some algebraic properties of the ideal $J$ attached to $\D$ through those of $I$ attached to $\C$. By the structure of $\D$, we see that $J$ is an extension of $I$ to  which shares some properties with $I$.  In case $\C$ admits a simplicial subclutter, it is guaranteed that  there exists $u\notin I$ such that adding $u$ to $I$ does not change  the nonlinear (multi)graded Betti numbers. Note that for any monomial ideal $I$ generated in degree $d$ and any monomial $u$ with $\deg(u)=d$, $\beta_{i, \textbf{\rm \textbf{a}}} \left( I +(u) \right) = \beta_{i, \textbf{\rm \textbf{a}}} \left( I \right)$, where $\textbf{\rm \textbf{a}} \in \mathbb{Z}^n$ with $|\textbf{\rm \textbf{a}}|> d+i+r$, and  where $r$ is the Castelnuovo-Mumford regularity of the ideal $I \colon u$, (c.f. Theorem~\ref{easy to verify}). But one can not expect more for a general case: Example~\ref{example 1} shows that the equality does not always hold in some multidegrees  $\mathrm{\mathbf{a}}$ with $|\textbf{\rm \textbf{a}}|= d+i+r$ even in case $r=1$ (i.e. $I\colon u$ is generated by variables).  However, in Theorem~\ref{inequality of multigraded betti numbers-thm}, we present a formula for the multigraded Betti numbers of the ideal $I+(u)$, with $u=\mathbf{x}_F$, in terms of those of $I$ and some numerical data about $F$, where $I$ and  $I+(u)$ are given respectively by a clutter $\C$ and its simplicial subclutter obtained by removing $F$.  As a result, one can see that the two ideals share the same regularity. In particular, $I$ has a linear resolution if and only if so does $I+(\mathbf{x}_F)$. 

 The proof of Theorem~\ref{inequality of multigraded betti numbers-thm} is a direct proof having topological and combinatorial flavor. However, after presenting this result at the CMS meeting 2017, Huy T\`ai H\`a mentioned to the first author that the ideal $I+(\mathbf{x}_F)$ is splittable, \cite{E-K}. This fact can be observed by the choice of $F$. On the other hand, since splittable ideals are  Betti splittable, \cite{FHVt}, one gets the Betti formula of Theorem~\ref{inequality of multigraded betti numbers-thm} from the Betti splitting formula. This short proof of Theorem~\ref{inequality of multigraded betti numbers-thm} is presented as well. The authors would like to thank Huy T\`ai H\`a for this remark. 

The final point in this section is that the ideal $I$ satisfies the subadditivity condition if and only if  so does $I+(\mathbf{x}_F)$.

\medskip
In Section~\ref{complete}, we consider the class of simplicial subclutters of the complete clutter $\C_{n,d}=\{F \subset [n] \colon |F|=d\}$. We call it the class of \textit{simplicial clutters}. This class is the dual  of a class called chordal clutters \cite{BYZ}. Roughly speaking, a chordal clutter, is a clutter which can be reduced to $\varnothing$ by some reduction operations, while a simplicial clutter is obtained from $\C_{n,d}$ by some reduction operations. As a consequence of Theorem~\ref{inequality of multigraded betti numbers-thm}, it is seen that the ideals attached to the class of chordal clutters have linear resolution over all fields. This class has been studied in \cite{BYZ}. Later, in \cite{MS} a modification of this variation of chordality was studied for simplicial complexes. Trying to compare the properties of chordal clutters and simplicial clutters, it turns out that, like chordal clutters, the ideal associated to a simplicial clutter has linear resolution over all fields  and all of its graded Betti numbers can be identified explicitly (Corollary~\ref{resolution of a simplicial subcluuter of complete clutter-cor}). Moreover, these ideals have linear quotients, while not all ideals of chordal clutters have this property, \cite[Example~3.14]{BYZ}. 

In  Theorem~\ref{stable ideals are simplicial subclutters of complete clutters} we see that the  square-free stable ideals are  examples of ideals associated to simplicial subclutters. 
In \cite{MJAR}, it is proved that any Betti sequence of an ideal with linear resolution is the Betti sequence of an ideal associated to a chordal clutter. One of the the main results of Section~\ref{complete}  improves this result
stating that the Betti sequence of any ideal with linear resolution coincides with the Betti sequence of an ideal associated to a simplicial clutter (see Theorem~\ref{all}). Hence all the Betti sequences of ideals with linear resolution are the Betti sequences of ideals with linear quotients.

One may ask if all the square-free equigenerated monomial ideals with linear quotients are attached to a simplicial clutter. In the last section, Section~\ref{44}, we give a negative answer to this question by presenting a simple example. On the other hand,  some known classes of ideals with linear resolution over all fields, such as square-free stable ideals, matroidal ideals and alexander dual of vertex decomposable ideals, are associated to  chordal clutters. Hence one may ask if all the square-free monomial ideals with linear resolution over all fields are attached to chordal clutters, \cite[Question~1]{BYZ}. We close  Section~\ref{44} showing by   some examples that, this is not the case. However,  to the best of our knowledge, this class is one of the largest known classes of clutters whose  associated ideal has linear resolution over every field.

\section{Preliminaries} \label{prelim}
Throughout this paper, $S=\mathbb{K}[x_1, \ldots, x_n]$ denotes the polynomial ring over the  field $\mathbb{K}$ endowed with multigraded standard grading; that is $\deg(x_i) = e_i$, where $e_i$s are the standard basis of $\mathbb{R}^n$.

\subsection{Multigraded Betti numbers}
Let $I \subset S$ be a monomial ideal in the polynomial ring $S$. Then $I$ is a multigraded $S$-module and so it admits a minimal multigraded free $S$-resolution
$$\mathcal{F}\colon \cdots \to F_2 \to F_1 \to F_0 \to I \to 0,$$
where $F_i = \bigoplus_{\textbf{\rm \textbf{a}} \in \mathbb{Z}^n} S \left( -\textbf{\rm \textbf{a}} \right)^{\beta_{i,\textbf{\rm \textbf{a}}}(I)}$. The numbers $\beta_{i,\textbf{\rm \textbf{a}}}(I) = \dim_\mathbb{K} \mathrm{Tor}^S_i \left( \mathbb{K}, I \right)_{\textbf{\rm \textbf{a}}}$ are called the \textit{multigraded Betti numbers of} $I$.

For $\textbf{\rm \textbf{a}}=(a_1, \ldots, a_n) \in \mathbb{Z}^n$, let $\mathrm{supp}(\textbf{\rm \textbf{a}}) = \{i \colon\, a_i \neq 0 \}$ and $|\textbf{\rm \textbf{a}}|= \sum_i a_i$. Define
$$\beta_{i,j}(I) = \sum\limits_{\textbf{\rm \textbf{a}} \in \mathbb{Z}^n \atop |\textbf{\rm \textbf{a}}|=j} \beta_{i, \textbf{\rm \textbf{a}}} (I).$$
The numbers $\beta_{i,j} (I)$ are the \textit{graded Betti numbers of} $I$ with respect to the standard grading on $S$ (i.e. $\deg(x_i)=1$). The \textit{Castelnuovo-Mumford regularity} of $I$, $\mathrm{reg} (I)$, is defined as follows:
$$\mathrm{reg} (I) = \max \{ j-i \colon \; \beta_{i,j} (I) \neq 0 \}.$$
A non-zero homogeneous ideal $I \subset S$ is said to have a $d$-linear resolution if $\beta_{i,j} (I) =0$, for all $i,j$ with $j-i>d$. In this case $I$ is generated by homogeneous elements of degree $d=\mathrm{reg}(I)$. A big class of ideals with linear resolution is the class of equigenerated ideals which have linear quotients,  \cite[Proposition~8.2.1]{HHBook}. An ideal $I$ is said to have \textit{linear quotients}, if $I$ has an ordered set of minimal generators $\left\{ u_1, \ldots, u_r \right\}$ such that the colon ideals $\left( u_1, \ldots, u_{i-1} \right) \colon u_i$ are generated by linear forms, for $i= 2, \ldots, r$.

\medspace
In general, adding an arbitrary monomial $u$ to a monomial ideal $I$ changes almost all of the Betti numbers. However, if $\deg u$  equals the degree of the generators of $I$, then we have

\begin{thm} \label{easy to verify}
Let $I \subset S$ be a monomial ideal  generated by elements in degree $d$ and let $u$ be a monomial in $S$ of degree $d$ with $u \not\in I$. Then
\begin{equation*}
\beta_{i, \textbf{\rm \textbf{a}}} \left( I +(u) \right) = \beta_{i, \textbf{\rm \textbf{a}}} \left( I \right),
\end{equation*}
for all $\textbf{\rm \textbf{a}} \in \mathbb{Z}^n$ with $|\textbf{\rm \textbf{a}}|> d+i+r$, where  $r= \mathrm{reg} (I \colon u)$. 
\end{thm}

\begin{proof}
Let $u= {\textbf{\rm \textbf{x}}}^{\bm \alpha}$, where $\bm\alpha\in \mathbb{Z}^n_{\geq 0}$. The short exact sequence 
\begin{equation*}
0 \to \frac{S}{I\colon u} (-{\bm \alpha}) \stackrel{.u}{\longrightarrow} \frac{S}{I} \longrightarrow \frac{S}{I+(u)} \to 0
\end{equation*}
of graded rings and homomorphisms induces the long exact sequence
\begin{equation} \label{long exact seq}
\begin{split}
\cdots \to \mathrm{Tor}_i^S \left( \mathbb{K}, \frac{S}{I \colon u} \right)_{\textbf{\rm \textbf{a}} - {\bm \alpha}} \longrightarrow \mathrm{Tor}_i^S \left( \mathbb{K}, \frac{S}{I} \right)_{\textbf{\rm \textbf{a}}} \longrightarrow \mathrm{Tor}_i^S \left( \mathbb{K}, \frac{S}{I + (u)} \right)_{\textbf{\rm \textbf{a}}}\\
\longrightarrow \mathrm{Tor}_{i-1}^S \left( \mathbb{K}, \frac{S}{I \colon u} \right)_{\textbf{\rm \textbf{a}} - {\bm \alpha}} \to  \cdots
\end{split}
\end{equation}
 for arbitrary $\textbf{\rm \textbf{a}} \in \mathbb{Z}^n$. 
% If $\mathrm{supp} (\bm{\alpha}) \not\subset \mathrm{supp}(\textbf{\rm \textbf{a}})$, then $\textbf{\rm \textbf{a}} - \bm{\alpha}$ has a negative entry and so $ \mathrm{Tor}_i^S \left( \mathbb{K}, {S}/({I \colon u}) \right)_{\textbf{\rm \textbf{a}} - {\bm \alpha}} =0$ for all $i$. So from the exact sequence (\ref{long exact seq}) we get
%\begin{equation} \label{negative entries}
%\beta_{i, \textbf{\rm \textbf{a}}} (\frac{S}{I}) = \dim_{\mathbb{K}} \mathrm{Tor}_i^S \left( \mathbb{K}, \frac{S}{I} \right)_{\textbf{\rm \textbf{a}}} = \dim_{\mathbb{K}} \mathrm{Tor}_i^S \left( \mathbb{K}, \frac{S}{I + (u)} \right)_{\textbf{\rm \textbf{a}}} = \beta_{i, \textbf{\rm \textbf{a}}} (\frac{S}{I+(u)}).
%\end{equation}
%Equation (\ref{negative entries}) holds if $\textbf{\rm \textbf{a}}$ has a negative entry, too. 
%So assume that $\mathrm{supp} (\bm{\alpha}) \subseteq \mathrm{supp}(\textbf{\rm \textbf{a}})$ and $\textbf{\rm \textbf{a}}$ has no  negative entry. 
%
%Now, 
If $|\textbf{\rm \textbf{a}}|>d+i+r-1$, then $|\textbf{\rm \textbf{a}} - \bm{\alpha}|>i +r-1$ and hence
$$\beta_{i-1, \textbf{\rm \textbf{a}} - {\bm \alpha}} \left( \frac{S}{I \colon u}  \right) = 0 = \beta_{i, \textbf{\rm \textbf{a}} - {\bm \alpha}} \left( \frac{S}{I \colon u} \right),$$
because $\mathrm{reg} \left( S/(I \colon u) \right) = \mathrm{reg} \left( I \colon u \right) -1 =r-1.$ It follows that
$$\mathrm{Tor}_{i-1}^S \left( \mathbb{K}, \frac{S}{I \colon u} \right)_{\textbf{\rm \textbf{a}} - {\bm \alpha}}  =  \mathrm{Tor}_i^S \left( \mathbb{K}, \frac{S}{I \colon u} \right)_{\textbf{\rm \textbf{a}} - {\bm \alpha}} =0,$$ 
and from the exact sequence (\ref{long exact seq}), we obtain $\mathrm{Tor}_i^S \left(\mathbb{K}, {S}/{I} \right)_{\textbf{\rm \textbf{a}}} \cong \mathrm{Tor}_i^S \left( \mathbb{K}, {S}/({I + (u)}) \right)_{\textbf{\rm \textbf{a}}}$ for all $\textbf{\rm \textbf{a}} \in \mathbb{Z}^n$ with $|\textbf{\rm \textbf{a}}|> d+i+r-1$. This implies that
$$\beta_{i, \textbf{\rm \textbf{a}}} (\frac{S}{I})= \dim_{\mathbb{K}} \mathrm{Tor}_i^S \left( \mathbb{K}, \frac{S}{I} \right)_{\textbf{\rm \textbf{a}}} = \dim_{\mathbb{K}} \mathrm{Tor}_i^S \left( \mathbb{K},  \frac{S}{I + (u)} \right)_{\textbf{\rm \textbf{a}}} = \beta_{i, \textbf{\rm \textbf{a}}} (\frac{S}{I+(u)}),$$
for all $\textbf{\rm \textbf{a}} \in \mathbb{Z}^n$ with $|\textbf{\rm \textbf{a}}|> d+i+r-1$. 
The assertion follows from the fact that $\beta_{i, \textbf{\rm \textbf{a}}} (I)= \beta_{i+1, \textbf{\rm \textbf{a}}} (S/I)$.
\end{proof}

As the following example shows,  in order to have the equality of multigraded Betti numbers in general, the lower bound for $|\textbf{\rm \textbf{a}}|$ in Theorem~\ref{easy to verify} is the best possible, even if $\mathrm{reg}(I \colon (u)) =1$, (i.e. $I \colon (u)$ is generated by a subset of variables).

\begin{ex} \label{example 1}
Let $I=(x_1x_4x_5,x_2x_3x_5)\subset \mathbb{K}[x_1,\ldots,x_5]$. Let $u=x_3x_4x_5$. Then $I\colon u=(x_1,x_2)$ and hence $\reg(I\colon u)=1$. Let 
$\textbf{\rm \textbf{a}}=(1,1,1,1,1)$. We have $|\textbf{\rm \textbf{a}}|=3+1+\reg(I\colon u)$. It is seen that $I+(u)$  has a $3$-linear resolution (hence $\beta_{1,\textbf{\rm \textbf{a}}}(I+(u))=0$) while $\beta_{1,\textbf{\rm \textbf{a}}} (I) =1$.  
\end{ex}

\subsection{Splittable monomial ideals}
A class of ideals whose Betti numbers can be computed through those of some of its subideals is the class of splittable ideals. Here we recall the definition.

For a monomial ideal $I \subset S$, let $\mathcal{G}(I)$ denote the (unique) set of minimal generators of $I$.

\begin{defn}[\cite{E-K}]
A monomial ideal $J$ is called  \textit{splittable} if $J$ is the sum of two non-zero monomial ideals $I$ and $K$, such that
\begin{itemize}
\item[(a)] $\mathcal{G}(J)$ is the disjoint union of $\mathcal{G}(I)$ and $\mathcal{G}(K)$;
\item[(b)] there is a splitting function
\begin{align*}
\mathcal{G}(I \cap K) & \to \mathcal{G}(I) \times \mathcal{G}(K) \\
w & \mapsto \left( \phi(w), \psi(w) \right)
\end{align*}
satisfying
\begin{enumerate}
\item for all $w \in \mathcal{G}(I \cap K)$, $w = \mathrm{lcm}\left( \phi(w), \psi(w) \right)$,
\item  for every subset $S \subseteq \mathcal{G}(I \cap K)$, both $\mathrm{lcm}(\phi(S))$ and $\mathrm{lcm}(\psi(S))$ strictly divide $\mathrm{lcm}(S)$.
\end{enumerate}
\end{itemize}
If $I$ and $K$ satisfy the above properties, then we say $J = I +K$ is a \textit{splitting} of $J$.
\end{defn}

\begin{defn}[{\cite{FHVt}}]
Let $I$, $J$, and $K$ be monomial ideals such that $\mathcal{G}(J)$ is the disjoint union of $\mathcal{G}(I)$ and $\mathcal{G}(K)$. Then $J = I + K$ is a \textit{Betti splitting} if
\begin{equation} \label{splittable ideals are betti splittable}
\beta_{i,j}(J) = \beta_{i,j}(I) + \beta_{i,j}(K) + \beta_{i-1,j}(I \cap K),
\end{equation}
for all $i \in \mathbb{Z}$ and (multi)degrees $j$.
\end{defn}

When $J = I +K$ is a splitting, Eliahou and Kervaire proved in \cite[Proposition 3.1]{E-K} that  $J = I +K$ is also Betti splittable.

Eliahou and Kervaire actually  proved (\ref{splittable ideals are betti splittable}) for total Betti numbers. Fatabbi \cite[Proposition 3.2]{Fatabi} extended the argument to
the graded case; in fact, her proof works as well if $j$ is a multidegree \cite{FHVt}.

\subsection{Simplicial complexes}
A \textit{simplicial complex} $\Delta$ on the vertex set $V=\{v_1, \ldots, v_n\}$ is a collection of subsets of $V$ such that $\{ v_i \} \in \Delta$  for all $i$ and, $F \in \Delta$ implies that all subsets of $F$ are also in $\Delta$. The elements of $\Delta$ are called
\textit{faces} and the maximal faces under inclusion are called \textit{facets} of $\Delta$. We denote by $\mathcal{F}(\Delta)$ the set of facets of $\Delta$. We say that a simplicial complex is {\em pure} if all its facets have the same cardinality. 
 The {\em dimension} of a face $F$ is $\dim F = |F|-1$, where $|F|$ denotes the cardinality of $F$. A simplicial complex is called {\em pure} if all its facets have the same dimension. 
The \textit{dimension} of $\Delta$, $\dim(\Delta)$, is defined as:
$$\dim (\Delta) = \max\{\dim F \colon F \in \Delta \}.$$

 A simplicial complex $\Gamma$ is called a {\em subcomplex} of $\Delta$ if all elements of $\Gamma$ belong to $\Delta$. If $W$ is a subset of $V$, the \textit{induced subcomplex} of $\Delta$ on $W$ is defined as $\Delta_W = \{ F \in \Delta \colon\; F \subseteq W\}$.

 Let $\Delta$ be a $d$-dimensional simplicial complex on $V$. For each $0\leq i\leq d$ the $i$-th {\em skeleton} of $\Delta$ is the simplicial complex $\Delta^{(i)}$ on $V$ whose faces are those faces $F$ of $\Delta$ with $\dim F \leq i$. In addition, we define  the {\em pure $i$-th skeleton} of $\Delta$ to be the pure subcomplex $\Delta^{[i]}$ of $\Delta$ whose facets are those faces $F$ of $\Delta$ with $\dim F = i$.

Given a simplicial complex $\Delta$ on the vertex set $\{v_{1}, \ldots, v_{n}\}$. For $F \subseteq \{v_{1}, \ldots, v_{n} \}$ let $\textbf{x}_F=\prod_{v_i \in F}{x_i}$, and let $\textbf{x}_\varnothing = 1$. 
The \emph{non-face ideal} or the \emph{Stanley-Reisner ideal} of $\Delta$, denoted by $I_\Delta$, is an ideal of $S$ generated by square-free monomials $\textbf{x}_F$, where $F \not\in \Delta$.

Let $\widetilde{H}_i \left( \Delta; \mathbb{K} \right)$ denote the $i$-th reduced homology of the augmented oriented chain complex of a simplicial complex $\Delta$ with  coefficients in the field $ \mathbb{K}$. If $\Delta$ is a simplicial complex and $\Delta_1$ and $\Delta_2$ are subcomplexes of $\Delta$, then there is an exact sequence
\begin{align} \label{Reduced Mayer-Vietoris sequence}
\cdots \to \widetilde{H}_j(\Delta_1 \cap \Delta_2 ;  \mathbb{K}) \to \widetilde{H}_j(\Delta_1;  \mathbb{K}) \oplus \widetilde{H}_j(\Delta_2;  \mathbb{K}) \to  \widetilde{H}_j(\Delta_1 \cup \Delta_2 ;  \mathbb{K}) \to \nonumber \\
\to  \widetilde{H}_{j-1}(\Delta_1 \cap \Delta_2;  \mathbb{K}) \to \cdots
\end{align}
 called the \textit{reduced Mayer-Vietoris sequence} of $\Delta_1$ and $\Delta_2$ (see \cite[Proposition 5.1.8]{HHBook} for more details). This exact sequence plays a crucial role in the proof of Proposition~\ref{preliminary statement} and Theorem~\ref{inequality of multigraded betti numbers-thm}.

Let $\Delta_1$ and $\Delta_2$ be simplicial complexes on disjoint vertex sets $V$ and $W$, respectively. The \textit{join} $\Delta_1 * \Delta_2$ is a simplicial complex on the vertex set $V \cup W$ with faces $F \cup G$, where $F \in \Delta_1$ and $G \in \Delta_2$. A {\em simplex} on the vertex set $V \neq \varnothing$ is the join of all  vertices in $V$ and it follows from \cite[Proposition 5.2.5]{Vil} that, if either $\Delta_1$ or $\Delta_2$ is a simplex, then $\widetilde{H}_j \left( \Delta_1 * \Delta_2; \mathbb{K} \right) = 0$ for all $j$.

\subsection{Clutters and their associated ideal}
In this section we recall some definitions about clutters and their associated ideal in the polynomial ring. 
\begin{defn}[Clutter] \label{SC} 
A \textit{clutter} $\mathcal{C}$ with the vertex set $[n]=\{ 1, \ldots, n\}$ is a collection of subsets of $[n]$, called \textit{circuits} of $\mathcal{C}$, such that if $F_1$ and $F_2$ are distinct circuits, then $F_1 \nsubseteq F_2$.
A \textit{$d$-circuit} is a circuit consisting of exactly $d$ vertices, and a clutter is called \textit{$d$-uniform}, if every circuit has $d$ vertices.
\end{defn}

For a non-empty clutter $\mathcal{C}$ on the vertex set $[n]$, we define the ideal $I \left( \mathcal{C} \right)$ to be
$$I(\mathcal{C}) = \left(  \textbf{x}_T \colon \quad T \in \mathcal{C} \right),$$
and we  set $I(\varnothing) = 0$. The ideal $I \left( {\mathcal{C}} \right)$ is called the \textit{circuit ideal} of $\mathcal{C}$. 

Let $n$, $d$ be positive integers and let $V$ be a set consisting $n$ elements. For $n \geq d$, by  $\mathcal{C}_{n,d}$ on the vertex set $V$  we mean the clutter
$$\mathcal{C}_{n,d} = \left\{ F \subset V \colon \quad |F|=d \right\}.$$
This clutter is called  the \textit{complete $d$-uniform clutter} on $V$ with $n$ vertices.
In the case that $n<d$, we let $\mathcal{C}_{n,d}$ be some isolated points. It is well-known that for $n \geq d$ the ideal $I \left( \mathcal{C}_{n,d} \right)$ has a $d$-linear resolution (see e.g. \cite[Example 2.12]{MNYZ}). 

If $\mathcal{C}$ is a $d$-uniform clutter on $[n]$, we define $\bar{\mathcal{C}}$, the \textit{complement} of $\mathcal{C}$, to be
\begin{equation*}
\bar{\mathcal{C}} = \mathcal{C}_{n,d} \setminus \mathcal{C} = \{F \subset [n] \colon \quad |F|=d, \,F \notin \mathcal{C}\}.
\end{equation*}

Frequently in this paper, we take a $d$-uniform clutter $\mathcal{C} \neq \mathcal{C}_{n,d}$ with vertex set $[n]$ and consider the square-free monomial ideal $I=I(\bar{\mathcal{C}})$ in the polynomial ring $S=\mathbb{\mathbb{K}}[x_1, \ldots, x_n]$. 

Let $\mathcal{C}$ be a clutter and let $e$ be a subset of $[n]$. By $\mathcal{C} \setminus e$ we mean the clutter 
$$ \left\{F \in \mathcal{C} \colon \ e \nsubseteq F \right\}.$$
This operation is called the \textit{deletion} of $e$ from $\mathcal{C}$.

\subsection{Chordal clutters}
In the following we recall some definitions and concepts from \cite{BYZ}.

%Let $\mathcal{C}$ be a $d$-uniform clutter on $[n]$. A subset $V \subset [n]$ is called a \textit{clique} in $\mathcal{C}$, if all $d$-subsets of $V$ belong to $\mathcal{C}$. A subset of $[n]$ with less than $d$ elements is considered to be a clique. 
%The set of cliques of $\mathcal{C}$ forms a simplicial complex, denoted $\Delta(\mathcal{C})$, which is called the {\em clique complex} of $\mathcal{C}$. Its Stanley--Reisner ideal $I_{\Delta(\mathcal{C})}$ coincides with the circuit ideal $I(\bar{\mathcal{C}})$ of $\bar{\mathcal{C}}$ \cite[Proposition 4.4]{MYZ}.

Let $\mathcal{C}$ be a $d$-uniform clutter on the vertex set $[n]$ and let $\Delta (\C)$ be the simplicial complex on the vertex set $[n]$ with $I_{\Delta (\C)} = I \left( \bar{\C} \right)$. The simplicial complex $\Delta (\C )$ is called the \textit{clique complex} of $\C$ and a face $F \in \Delta (\C)$ is called a \textit{clique} in $\C$. It is easily seen that $F \subset [n]$ is a clique in $\C$ if and only if either $|F|<d$ or else all $d$-subsets of $F$ belongs to $\C$ (see \cite[Definition 4.2 and Proposition 4.4]{MAR}).

Let $\mathcal{C}$ be a $d$-uniform clutter on the vertex set $[n]$. For any $(d-1)$-subset $e$ of $[n]$, let
\begin{equation*}
{N}_{\mathcal{C}} \left[ e \right] = e \cup \lbrace c \in \left[ n \right] \colon \quad e \cup \lbrace c \rbrace \in \mathcal{C} \rbrace.
\end{equation*}
We call ${N}_{\mathcal{C}} \left[ e \right]$ the \textit{closed neighborhood} of $e$ in $\mathcal{C}$.  In the case that $ N_{\mathcal{C}} \left[ e \right]\neq e$ (i.e. $e \subset F$ for some $F \in \mathcal{C}$), $e$ is called a \textit{maximal subcircuit} of $\mathcal{C}$. The set of all maximal subcircuits of $\mathcal{C}$ is denoted by ${\rm SC}\left(\mathcal{C}\right)$. We say that $e$ is \textit{simplicial} over $\mathcal{C}$ if ${N}_{\mathcal{C}} \left[ e \right] \in \Delta(\mathcal{C})$. One may note that a $(d-1)$-subset of $[n]$ which is not a maximal subcircuit of $\C$ is simplicial over $\mathcal{C}$, called trivial simplicial elements. If $e \in \mathrm{SC} \left(\mathcal{C}\right)$ and  $e$ is simplicial over $\mathcal{C}$, then $e$ is called a \textit{simplicial maximal subcircuit} of $\mathcal{C}$. Let us denote by $\mathrm{Simp} \left( \mathcal{C} \right)$ the set of all $(d-1)$-subsets of $[n]$ which are simplicial over $\mathcal{C}$.

\begin{defn}[{\cite[Definition 3.1]{BYZ}}] \label{def of chordal}
Let $\mathcal{C}$ be a $d$-uniform clutter. We call $\mathcal{C}$ a \textit{chordal clutter}, if either $\mathcal{C}=\varnothing$, or $\mathcal{C}$ admits a simplicial maximal subcircuit $e$ such that $\mathcal{C}\setminus e$ is chordal.
\end{defn}
Following the notation in \cite{BYZ}, we use $\mathfrak{C}_d$, to denote the class of all $d$-uniform chordal clutters.

Let $\mathcal{C}$ be a $d$-uniform clutter. A sequence $\eb= e_1, \ldots, e_t$ of  $(d-1)$-subsets of $[n]$, is called a \textit{simplicial sequence} in  $\mathcal{C}$ if $e_1$ is  simplicial over $\mathcal{C}$ and $e_i$ is simplicial over $\left( \left( \left( \mathcal{C} \setminus e_{1} \right) \setminus e_2 \right)  \setminus \cdots \right) \setminus e_{i-1}$ for all $i>1$.

One may rewrite the Definition~\ref{def of chordal} as follows: A $d$-uniform clutter $\mathcal{C}$ is chordal if either $\mathcal{C} = \varnothing$, or else there exists a simplicial sequence in  $\mathcal{C}$, say $\eb=e_1, \ldots, e_t$, such that $e_1$ is a maximal subcircuit over $\C$, $e_i$ is a maximal subcircuit over $\C \setminus e_1 \setminus \cdots \setminus e_{i-1}$ for $i>1$, and $\left( \left( \left(  \mathcal{C} \setminus e_{1} \right) \setminus e_2 \right) \setminus \cdots \right) \setminus e_{t} = \varnothing$.  The sequence $\eb$ is called a {\em simplicial order} of $\mathcal{C}$.

%To simplify the notation, given a $d$-uniform clutter $\mathcal{C}$ and a simplicial sequence $\textbf{e} = e_1, \ldots, e_r$ in $\mathcal{C}$, we use $\mathcal{C}_\textbf{e}^0$ for  $\mathcal{C}$ and $\mathcal{C}_\textbf{e}^i $ for $\left( \mathcal{C} \setminus e_1 \right) \setminus \cdots \setminus e_{i}$ for $i\geq 1$.

In \cite[Remark~2]{BYZ} it is mentioned that Definition~\ref{def of chordal} coincides with the graph theoretical definition of chordal graphs in the case $d=2$.

\begin{ex}
In Figure~\ref{F2}, the $3$-uniform clutter $\mathcal{C}$ is chordal, while  the $3$-uniform clutter $\mathcal{D}$ is not.
	\begin{align*}
	&\mathcal{C}=\{ \{ 1,2,3 \}, \{ 1,2,4 \}, \{ 1,3,4 \}, \{ 2,3,4 \}, \{ 1,2,5 \},
	\{ 1,2,6 \}, \{ 1,5,6 \}, \{ 2,5,6 \} \}. \\
	&\mathcal{D}=\{\{1,2,3\}, \{1,2,4\}, \{1,3,4\},\{2,3,5\}, \{2,4,5\}, \{3,4,5\}\}.
	\end{align*}
	\begin{figure}[ht!]
		\begin{tikzpicture}[line cap=round,line join=round,>=triangle 45,x=0.7cm,y=0.7cm]
		\clip(5.26,2.0) rectangle (12.44,6.8);
		\fill[fill=black,fill opacity=0.1] (8.8,5.94) -- (6.2,4.5) -- (7.66,3.74) -- cycle;
		\fill[fill=black,fill opacity=0.1] (8.8,5.94) -- (8.8,4.48) -- (7.66,3.74) -- cycle;
		\fill[fill=black,fill opacity=0.1] (8.8,5.94) -- (11.46,4.48) -- (9.98,3.74) -- cycle;
		\fill[fill=black,fill opacity=0.1] (8.8,5.94) -- (8.8,4.48) -- (9.98,3.74) -- cycle;
		\draw (8.8,5.94)-- (6.2,4.5);
		\draw (6.2,4.5)-- (7.66,3.74);
		\draw (7.66,3.74)-- (8.8,5.94);
		\draw (8.8,5.94)-- (8.8,4.48);
		\draw (8.8,4.48)-- (7.66,3.74);
		\draw (7.66,3.74)-- (8.8,5.94);
		\draw (8.8,5.94)-- (11.46,4.48);
		\draw (11.46,4.48)-- (9.98,3.74);
		\draw (9.98,3.74)-- (8.8,5.94);
		\draw (8.8,5.94)-- (8.8,4.48);
		\draw (8.8,4.48)-- (9.98,3.74);
		\draw (9.98,3.74)-- (8.8,5.94);
		\draw [dash pattern=on 2pt off 2pt] (6.2,4.5)-- (11.46,4.48);
		\draw (8.52,6.58) node[anchor=north west] {\begin{scriptsize}1\end{scriptsize}};
		\draw (8.44,4.50) node[anchor=north west] {\begin{scriptsize}2\end{scriptsize}};
		\draw (5.62,4.83) node[anchor=north west] {\begin{scriptsize}3\end{scriptsize}};
		\draw (7.4,3.8) node[anchor=north west] {\begin{scriptsize}4\end{scriptsize}};
		\draw (11.38,4.82) node[anchor=north west] {\begin{scriptsize}5\end{scriptsize}};
		\draw (9.74,3.8) node[anchor=north west] {\begin{scriptsize}6\end{scriptsize}};
		\begin{scriptsize}
		\fill [color=black] (8.8,5.94) circle (1.3pt);
		\fill [color=black] (6.2,4.5) circle (1.3pt);
		\fill [color=black] (7.66,3.74) circle (1.3pt);
		\fill [color=black] (8.8,4.48) circle (1.3pt);
		\fill [color=black] (11.46,4.48) circle (1.3pt);
		\fill [color=black] (9.98,3.74) circle (1.3pt);
		\end{scriptsize}
		\end{tikzpicture}
		\quad
		\quad
		\quad
		\begin{tikzpicture}[line cap=round,line join=round,>=triangle 45,x=.7cm,y=.7cm]
		\clip(13.89,2.00) rectangle (18.76,6.8);
		\fill[fill=black,fill opacity=0.1] (16.24,6.16) -- (15.14,4.48) -- (16.7,4.06) -- cycle;
		\fill[fill=black,fill opacity=0.1] (16.24,6.16) -- (17.34,4.68) -- (16.7,4.06) -- cycle;
		\fill[fill=black,fill opacity=0.1] (15.14,4.48) -- (16.28,2.62) -- (16.7,4.06) -- cycle;
		\fill[fill=black,fill opacity=0.1] (17.34,4.68) -- (16.28,2.62) -- (16.7,4.06) -- cycle;
		\draw (16.24,6.16)-- (15.14,4.48);
		\draw (15.14,4.48)-- (16.7,4.06);
		\draw (16.7,4.06)-- (16.24,6.16);
		\draw (16.24,6.16)-- (17.34,4.68);
		\draw (17.34,4.68)-- (16.7,4.06);
		\draw (16.7,4.06)-- (16.24,6.16);
		\draw [dash pattern=on 3pt off 3pt] (15.14,4.48)-- (17.34,4.68);
		\draw (15.14,4.48)-- (16.28,2.62);
		\draw (16.28,2.62)-- (16.7,4.06);
		\draw (16.7,4.06)-- (15.14,4.48);
		\draw (17.34,4.68)-- (16.28,2.62);
		\draw (16.28,2.62)-- (16.7,4.06);
		\draw (16.7,4.06)-- (17.34,4.68);
		\draw (15.95,6.83) node[anchor=north west] {\begin{scriptsize}1\end{scriptsize}};
		\draw (17.31,5) node[anchor=north west] {\begin{scriptsize}2\end{scriptsize}};
		\draw (14.55,4.8) node[anchor=north west] {\begin{scriptsize}3\end{scriptsize}};
		\draw (16.07,4.17) node[anchor=north west] {\begin{scriptsize}4\end{scriptsize}};
		\draw (15.95,2.62) node[anchor=north west] {\begin{scriptsize}5\end{scriptsize}};
		\begin{scriptsize}
		\fill [color=black] (16.24,6.16) circle (1.3pt);
		\fill [color=black] (15.14,4.48) circle (1.3pt);
		\fill [color=black] (16.7,4.06) circle (1.3pt);
		\fill [color=black] (17.34,4.68) circle (1.3pt);
		\fill [color=black] (16.28,2.62) circle (1.3pt);
		\end{scriptsize}
		\end{tikzpicture}
		\caption{The Clutter $\mathcal{C}$ on the left and $\mathcal{D}$ on the right}
		\label{F2}
	\end{figure}
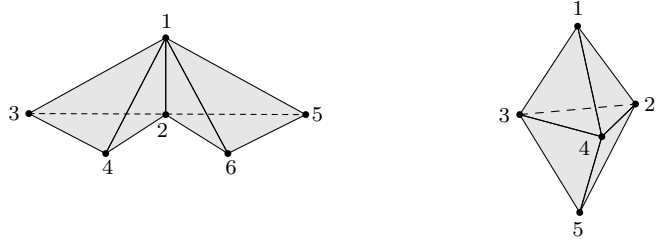

Indeed, while the clutter $\mathcal{D}$ does not have any simplicial maximal subcircuit, one of the possible simplicial orders for $\C$ is the following:
\begin{center}
\begin{tabular}{lll}
$e_1 = \{ 1, 3 \}$ & $e_2 = \{ 1, 4 \}$ & $e_3 = \{ 2, 4 \}$\\
$e_4 = \{ 1, 2 \}$ & $e_5 = \{ 2, 6 \}$ & $e_6 = \{ 1, 5 \}$\\
\end{tabular}
\end{center}
\end{ex}

\section{Simplicial subclutters and their multigraded resolution} \label{Main}
In this section, we consider a class of subclutters, called simplicial subclutters, which have the property that all the (multi)graded Betti numbers of their associated ideals  can be determined via those of the ideal of the original clutter.

\begin{defn} \label{Definition of simplicial subclutter}
Let $\mathcal{C}$ be a $d$-uniform clutter on the vertex set $[n]$ and let  $\mathcal{D} \subsetneq \mathcal{C}$. We say that $\mathcal{D}$ is a \textit{simplicial subclutter} of $\mathcal{C}$ if there exists a sequence of $(d-1)$-subsets of $[n]$ (not necessarily distinct), say $\textbf{e}= e_1, \ldots, e_t$, and $A_i \subseteq \left\{ F \in \mathcal{C}\setminus A_1\setminus \cdots\setminus A_{i-1} \colon \; e_i \subset F \right\}$, $i=1, \ldots, t$, such that 
\begin{itemize}
\item[(a)] $e_1$ is simplicial over $\mathcal{C}$;
\item[(b)] $e_i$ is simplicial over $\mathcal{C} \setminus A_1 \setminus \cdots \setminus A_{i-1}$, for $i>1$;
\item[(c)] $\mathcal{D} = \mathcal{C} \setminus A_1 \setminus \cdots \setminus A_t$.
\end{itemize}
\end{defn}

\begin{ex}
Let $G$ be the graph illustrated in Figure~\ref{graph0}. Let $e_1=\{1\}$. We have $\mathrm{N}_{G}[e_1]=\{1,2,3,5\}$. Any $2$-subset of this set forms an edge in $G$. Hence,  $N_G[e_1]\in \Delta(G)$ and so $\{1\}$ is a simplicial maximal subcircuit of $G$, called simplicial vertex of $G$. Let $A_1=\{\{1,3\},\{1,5\}\}$, and let $G_1=G\setminus A_1$.
	
Now let $e_2=\{6\}$. Since $\mathrm{N}_{G_1}[e_2]=\{5,6,7\}\in \Delta(G_1)$, the vertex $6$ is a simplicial vertex of $G_1$. Let $A_2=\{\{5,6\},\{5,7\}\}$ and let $G_2= G_1\setminus A_2$. Finally, let $e_3=\{8\}$. We have $\mathrm{N}_{G_2}[e_3]=\{7,8\}\in \Delta(G_2)$. So $8$ is a simplicial vertex of $G_2$.  Let $A_3=\{\{7,8\}\}$ and let $G_3= G_2\setminus A_3$. 
The graph $G_i$ is a simplicial subgraph of $G_j$ for $j<i$ and they all are simplicial subgraphs of $G$. All these graphs are shown in Figure~\ref{graph0}.
\begin{figure}[ht!]
	\begin{center}
		\begin{tikzpicture}[line cap=round,line join=round,>=triangle 45,x=0.6cm,y=0.6cm]
		\clip(2.5,2.5) rectangle (24,8.5);
		\fill[line width=0.8pt,fill=white,fill opacity=0.10000000149011612] (5.,8.) -- (3.12,6.66) -- (3.8134637824106052,4.4579309769026825) -- (6.1220479699074195,4.436977475055262) -- (6.855367681260407,6.626096521827541) -- cycle;
		\fill[line width=0.8pt,fill=white,fill opacity=0.10000000149011612] (6.855367681260407,6.626096521827541) -- (9.,8.) -- (9.04,5.64) -- cycle;
		\draw [line width=0.8pt] (5.,8.)-- (3.12,6.66);
		\draw [line width=0.8pt] (3.12,6.66)-- (3.8134637824106052,4.4579309769026825);
		\draw [line width=0.8pt] (3.8134637824106052,4.4579309769026825)-- (6.1220479699074195,4.436977475055262);
		\draw [line width=0.8pt] (6.1220479699074195,4.436977475055262)-- (6.855367681260407,6.626096521827541);
		\draw [line width=0.8pt] (6.855367681260407,6.626096521827541)-- (5.,8.);
		\draw [line width=0.8pt] (6.855367681260407,6.626096521827541)-- (9.,8.);
		\draw [line width=0.8pt] (9.,8.)-- (9.04,5.64);
		\draw [line width=0.8pt] (9.04,5.64)-- (6.855367681260407,6.626096521827541);
		\draw [line width=0.8pt] (5.,8.)-- (3.8134637824106052,4.4579309769026825);
		\draw [line width=0.8pt] (6.855367681260407,6.626096521827541)-- (3.8134637824106052,4.4579309769026825);
		\draw [line width=0.8pt] (6.855367681260407,6.626096521827541)-- (3.12,6.66);
		\draw [line width=0.8pt] (9.04,5.64)-- (10.94,4.42);
		\draw [line width=0.8pt] (9.04,5.64)-- (7.3,4.42);
		\draw (4.65,8.7) node[anchor=north west] {\begin{scriptsize}1\end{scriptsize}};
		\draw (2.55,7.05) node[anchor=north west] {\begin{scriptsize}2\end{scriptsize}};
		\draw (3.45,4.5) node[anchor=north west] {\begin{scriptsize}3\end{scriptsize}};
		\draw (5.75,4.5) node[anchor=north west] {\begin{scriptsize}4\end{scriptsize}};
		\draw (6.66,6.6) node[anchor=north west] {\begin{scriptsize}5\end{scriptsize}};
		\draw (9.04,6.1) node[anchor=north west] {\begin{scriptsize}7\end{scriptsize}};
		\draw (8.72,8.7) node[anchor=north west] {\begin{scriptsize}6\end{scriptsize}};
		\draw (7.,4.45) node[anchor=north west] {\begin{scriptsize}8\end{scriptsize}};
		\draw (10.68,4.45) node[anchor=north west] {\begin{scriptsize}9\end{scriptsize}};
		\draw (5.2,3.5) node[anchor=north west] {Graph $G$};
		\begin{scriptsize}
		\draw [fill=black] (5.,8.) circle (1.5pt);
		\draw [fill=black] (3.12,6.66) circle (1.5pt);
		\draw [fill=black] (3.8134637824106052,4.4579309769026825) circle (1.5pt);
		\draw [fill=black] (6.1220479699074195,4.436977475055262) circle (1.5pt);
		\draw [fill=black] (6.855367681260407,6.626096521827541) circle (1.5pt);
		\draw [fill=black] (9.,8.) circle (1.5pt);
		\draw [fill=black] (9.04,5.64) circle (1.5pt);
		\draw [fill=black] (10.94,4.42) circle (1.5pt);
		\draw [fill=black] (7.3,4.42) circle (1.5pt);
		\end{scriptsize}
		\hspace{7.5cm}
		%\clip(2.5,3.9) rectangle (11.7,8.5);
		\fill[line width=0.8pt,fill=white,fill opacity=0.10000000149011612] (5.,8.) -- (3.12,6.66) -- (3.8134637824106052,4.4579309769026825) -- (6.1220479699074195,4.436977475055262) -- (6.855367681260407,6.626096521827541) -- cycle;
		\fill[line width=0.8pt,fill=white,fill opacity=0.10000000149011612] (6.855367681260407,6.626096521827541) -- (9.,8.) -- (9.04,5.64) -- cycle;
		\draw [line width=0.8pt] (5.,8.)-- (3.12,6.66);
		\draw [line width=0.8pt] (3.12,6.66)-- (3.8134637824106052,4.4579309769026825);
		\draw [line width=0.8pt] (3.8134637824106052,4.4579309769026825)-- (6.1220479699074195,4.436977475055262);
		\draw [line width=0.8pt] (6.1220479699074195,4.436977475055262)-- (6.855367681260407,6.626096521827541);
		%\draw [line width=0.8pt] (6.855367681260407,6.626096521827541)-- (5.,8.);
		\draw [line width=0.8pt] (6.855367681260407,6.626096521827541)-- (9.,8.);
		\draw [line width=0.8pt] (9.,8.)-- (9.04,5.64);
		\draw [line width=0.8pt] (9.04,5.64)-- (6.855367681260407,6.626096521827541);
		%\draw [line width=0.8pt] (5.,8.)-- (3.8134637824106052,4.4579309769026825);
		\draw [line width=0.8pt] (6.855367681260407,6.626096521827541)-- (3.8134637824106052,4.4579309769026825);
		\draw [line width=0.8pt] (6.855367681260407,6.626096521827541)-- (3.12,6.66);
		\draw [line width=0.8pt] (9.04,5.64)-- (10.94,4.42);
		\draw [line width=0.8pt] (9.04,5.64)-- (7.3,4.42);
		\draw (4.65,8.7) node[anchor=north west] {\begin{scriptsize}1\end{scriptsize}};
		\draw (2.55,7.05) node[anchor=north west] {\begin{scriptsize}2\end{scriptsize}};
		\draw (3.45,4.5) node[anchor=north west] {\begin{scriptsize}3\end{scriptsize}};
		\draw (5.75,4.5) node[anchor=north west] {\begin{scriptsize}4\end{scriptsize}};
		\draw (6.66,6.6) node[anchor=north west] {\begin{scriptsize}5\end{scriptsize}};
		\draw (9.04,6.1) node[anchor=north west] {\begin{scriptsize}7\end{scriptsize}};
		\draw (8.72,8.7) node[anchor=north west] {\begin{scriptsize}6\end{scriptsize}};
		\draw (7,4.45) node[anchor=north west] {\begin{scriptsize}8\end{scriptsize}};
		\draw (10.68,4.45) node[anchor=north west] {\begin{scriptsize}9\end{scriptsize}};
		\draw (5.2,3.5) node[anchor=north west] {Graph $G_1$};
		\begin{scriptsize}
		\draw [fill=black] (5.,8.) circle (1.5pt);
		\draw [fill=black] (3.12,6.66) circle (1.5pt);
		\draw [fill=black] (3.8134637824106052,4.4579309769026825) circle (1.5pt);
		\draw [fill=black] (6.1220479699074195,4.436977475055262) circle (1.5pt);
		\draw [fill=black] (6.855367681260407,6.626096521827541) circle (1.5pt);
		\draw [fill=black] (9.,8.) circle (1.5pt);
		\draw [fill=black] (9.04,5.64) circle (1.5pt);
		\draw [fill=black] (10.94,4.42) circle (1.5pt);
		\draw [fill=black] (7.3,4.42) circle (1.5pt);
		\end{scriptsize}
		\end{tikzpicture}

		\begin{tikzpicture}[line cap=round,line join=round,>=triangle 45,x=0.6cm,y=0.6cm]
		\clip(2.5,2.5) rectangle (24,8.5);
		\fill[line width=0.8pt,fill=white,fill opacity=0.10000000149011612] (5.,8.) -- (3.12,6.66) -- (3.8134637824106052,4.4579309769026825) -- (6.1220479699074195,4.436977475055262) -- (6.855367681260407,6.626096521827541) -- cycle;
		\fill[line width=0.8pt,fill=white,fill opacity=0.10000000149011612] (6.855367681260407,6.626096521827541) -- (9.,8.) -- (9.04,5.64) -- cycle;
		\draw [line width=0.8pt] (5.,8.)-- (3.12,6.66);
		\draw [line width=0.8pt] (3.12,6.66)-- (3.8134637824106052,4.4579309769026825);
		\draw [line width=0.8pt] (3.8134637824106052,4.4579309769026825)-- (6.1220479699074195,4.436977475055262);
		\draw [line width=0.8pt] (6.1220479699074195,4.436977475055262)-- (6.855367681260407,6.626096521827541);
		%\draw [line width=0.8pt] (6.855367681260407,6.626096521827541)-- (5.,8.);
		%\draw [line width=0.8pt] (6.855367681260407,6.626096521827541)-- (9.,8.);
		%\draw [line width=0.8pt] (9.,8.)-- (9.04,5.64);
		\draw [line width=0.8pt] (9.04,5.64)-- (6.855367681260407,6.626096521827541);
		%\draw [line width=0.8pt] (5.,8.)-- (3.8134637824106052,4.4579309769026825);
		\draw [line width=0.8pt] (6.855367681260407,6.626096521827541)-- (3.8134637824106052,4.4579309769026825);
		\draw [line width=0.8pt] (6.855367681260407,6.626096521827541)-- (3.12,6.66);
		\draw [line width=0.8pt] (9.04,5.64)-- (10.94,4.42);
		\draw [line width=0.8pt] (9.04,5.64)-- (7.3,4.42);
		\draw (4.65,8.7) node[anchor=north west] {\begin{scriptsize}1\end{scriptsize}};
		\draw (2.55,7.05) node[anchor=north west] {\begin{scriptsize}2\end{scriptsize}};
		\draw (3.45,4.5) node[anchor=north west] {\begin{scriptsize}3\end{scriptsize}};
		\draw (5.75,4.5) node[anchor=north west] {\begin{scriptsize}4\end{scriptsize}};
		\draw (6.66,6.6) node[anchor=north west] {\begin{scriptsize}5\end{scriptsize}};
		\draw (9.04,6.1) node[anchor=north west] {\begin{scriptsize}7\end{scriptsize}};
		\draw (8.72,8.7) node[anchor=north west] {\begin{scriptsize}6\end{scriptsize}};
		\draw (7,4.45) node[anchor=north west] {\begin{scriptsize}8\end{scriptsize}};
		\draw (10.68,4.45) node[anchor=north west] {\begin{scriptsize}9\end{scriptsize}};
		\draw (5.2,3.5) node[anchor=north west] {Graph $G_2$};
		\begin{scriptsize}
		\draw [fill=black] (5.,8.) circle (1.5pt);
		\draw [fill=black] (3.12,6.66) circle (1.5pt);
		\draw [fill=black] (3.8134637824106052,4.4579309769026825) circle (1.5pt);
		\draw [fill=black] (6.1220479699074195,4.436977475055262) circle (1.5pt);
		\draw [fill=black] (6.855367681260407,6.626096521827541) circle (1.5pt);
		\draw [fill=black] (9.,8.) circle (1.5pt);
		\draw [fill=black] (9.04,5.64) circle (1.5pt);
		\draw [fill=black] (10.94,4.42) circle (1.5pt);
		\draw [fill=black] (7.3,4.42) circle (1.5pt);
		\end{scriptsize}
		\hspace{7.5cm}
		\fill[line width=0.8pt,fill=white,fill opacity=0.10000000149011612] (5.,8.) -- (3.12,6.66) -- (3.8134637824106052,4.4579309769026825) -- (6.1220479699074195,4.436977475055262) -- (6.855367681260407,6.626096521827541) -- cycle;
		\fill[line width=0.8pt,fill=white,fill opacity=0.10000000149011612] (6.855367681260407,6.626096521827541) -- (9.,8.) -- (9.04,5.64) -- cycle;
		\draw [line width=0.8pt] (5.,8.)-- (3.12,6.66);
		\draw [line width=0.8pt] (3.12,6.66)-- (3.8134637824106052,4.4579309769026825);
		\draw [line width=0.8pt] (3.8134637824106052,4.4579309769026825)-- (6.1220479699074195,4.436977475055262);
		\draw [line width=0.8pt] (6.1220479699074195,4.436977475055262)-- (6.855367681260407,6.626096521827541);
		%\draw [line width=0.8pt] (6.855367681260407,6.626096521827541)-- (5.,8.);
		%\draw [line width=0.8pt] (6.855367681260407,6.626096521827541)-- (9.,8.);
		%\draw [line width=0.8pt] (9.,8.)-- (9.04,5.64);
		\draw [line width=0.8pt] (9.04,5.64)-- (6.855367681260407,6.626096521827541);
		%\draw [line width=0.8pt] (5.,8.)-- (3.8134637824106052,4.4579309769026825);
		\draw [line width=0.8pt] (6.855367681260407,6.626096521827541)-- (3.8134637824106052,4.4579309769026825);
		\draw [line width=0.8pt] (6.855367681260407,6.626096521827541)-- (3.12,6.66);
		\draw [line width=0.8pt] (9.04,5.64)-- (10.94,4.42);
		%\draw [line width=0.8pt] (9.04,5.64)-- (7.3,4.42);
		\draw (4.65,8.7) node[anchor=north west] {\begin{scriptsize}1\end{scriptsize}};
		\draw (2.55,7.05) node[anchor=north west] {\begin{scriptsize}2\end{scriptsize}};
		\draw (3.45,4.5) node[anchor=north west] {\begin{scriptsize}3\end{scriptsize}};
		\draw (5.75,4.5) node[anchor=north west] {\begin{scriptsize}4\end{scriptsize}};
		\draw (6.66,6.6) node[anchor=north west] {\begin{scriptsize}5\end{scriptsize}};
		\draw (9.04,6.1) node[anchor=north west] {\begin{scriptsize}7\end{scriptsize}};
		\draw (8.72,8.7) node[anchor=north west] {\begin{scriptsize}6\end{scriptsize}};
		\draw (7,4.45) node[anchor=north west] {\begin{scriptsize}8\end{scriptsize}};
		\draw (10.68,4.45) node[anchor=north west] {\begin{scriptsize}9\end{scriptsize}};
		\draw (5.2,3.5) node[anchor=north west] {Graph $G_3$};
		\begin{scriptsize}
		\draw [fill=black] (5.,8.) circle (1.5pt);
		\draw [fill=black] (3.12,6.66) circle (1.5pt);
		\draw [fill=black] (3.8134637824106052,4.4579309769026825) circle (1.5pt);
		\draw [fill=black] (6.1220479699074195,4.436977475055262) circle (1.5pt);
		\draw [fill=black] (6.855367681260407,6.626096521827541) circle (1.5pt);
		\draw [fill=black] (9.,8.) circle (1.5pt);
		\draw [fill=black] (9.04,5.64) circle (1.5pt);
		\draw [fill=black] (10.94,4.42) circle (1.5pt);
		\draw [fill=black] (7.3,4.42) circle (1.5pt);
		\end{scriptsize}
		\end{tikzpicture}
		\caption{The graph $G$ and three of its simplicial subgraphs}
		\label{graph0}
	\end{center}
\end{figure}
One can continue this process by removing some edges containing simplicial vertices in order to have more simplicial subgraphs of $G$.

But the subgraph $G'$ of $G$ shown in Figure~\ref{G''} is not a simplicial subgraph of $G$, because $G'$ is obtained from $G$ by removing the edge $\{5,7\}$ while none of the vertices $5, 7$ is simplicial.

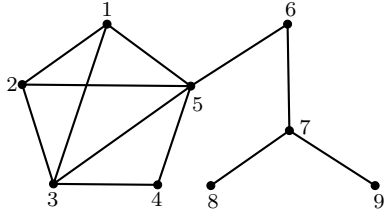
\begin{figure}[ht!]
	\begin{center}
		\begin{tikzpicture}[line cap=round,line join=round,>=triangle 45,x=0.6cm,y=0.6cm]
		\clip(2.5,3.9) rectangle (11.7,8.5);
		\fill[line width=0.8pt,fill=white,fill opacity=0.10000000149011612] (5.,8.) -- (3.12,6.66) -- (3.8134637824106052,4.4579309769026825) -- (6.1220479699074195,4.436977475055262) -- (6.855367681260407,6.626096521827541) -- cycle;
		\fill[line width=0.8pt,fill=white,fill opacity=0.10000000149011612] (6.855367681260407,6.626096521827541) -- (9.,8.) -- (9.04,5.64) -- cycle;
		\draw [line width=0.8pt] (5.,8.)-- (3.12,6.66);
		\draw [line width=0.8pt] (3.12,6.66)-- (3.8134637824106052,4.4579309769026825);
		\draw [line width=0.8pt] (3.8134637824106052,4.4579309769026825)-- (6.1220479699074195,4.436977475055262);
		\draw [line width=0.8pt] (6.1220479699074195,4.436977475055262)-- (6.855367681260407,6.626096521827541);
		\draw [line width=0.8pt] (6.855367681260407,6.626096521827541)-- (5.,8.);
		\draw [line width=0.8pt] (6.855367681260407,6.626096521827541)-- (9.,8.);
		\draw [line width=0.8pt] (9.,8.)-- (9.04,5.64);
		%\draw [line width=0.8pt] (9.04,5.64)-- (6.855367681260407,6.626096521827541);
		\draw [line width=0.8pt] (5.,8.)-- (3.8134637824106052,4.4579309769026825);
		\draw [line width=0.8pt] (6.855367681260407,6.626096521827541)-- (3.8134637824106052,4.4579309769026825);
		\draw [line width=0.8pt] (6.855367681260407,6.626096521827541)-- (3.12,6.66);
		\draw [line width=0.8pt] (9.04,5.64)-- (10.94,4.42);
		\draw [line width=0.8pt] (9.04,5.64)-- (7.3,4.42);
		\draw (4.65,8.7) node[anchor=north west] {\begin{scriptsize}1\end{scriptsize}};
		\draw (2.55,7.05) node[anchor=north west] {\begin{scriptsize}2\end{scriptsize}};
		\draw (3.45,4.5) node[anchor=north west] {\begin{scriptsize}3\end{scriptsize}};
		\draw (5.75,4.5) node[anchor=north west] {\begin{scriptsize}4\end{scriptsize}};
		\draw (6.66,6.6) node[anchor=north west] {\begin{scriptsize}5\end{scriptsize}};
		\draw (9.04,6.1) node[anchor=north west] {\begin{scriptsize}7\end{scriptsize}};
		\draw (8.72,8.7) node[anchor=north west] {\begin{scriptsize}6\end{scriptsize}};
		\draw (7,4.45) node[anchor=north west] {\begin{scriptsize}8\end{scriptsize}};
		\draw (10.68,4.45) node[anchor=north west] {\begin{scriptsize}9\end{scriptsize}};
		\begin{scriptsize}
		\draw [fill=black] (5.,8.) circle (1.5pt);
		\draw [fill=black] (3.12,6.66) circle (1.5pt);
		\draw [fill=black] (3.8134637824106052,4.4579309769026825) circle (1.5pt);
		\draw [fill=black] (6.1220479699074195,4.436977475055262) circle (1.5pt);
		\draw [fill=black] (6.855367681260407,6.626096521827541) circle (1.5pt);
		\draw [fill=black] (9.,8.) circle (1.5pt);
		\draw [fill=black] (9.04,5.64) circle (1.5pt);
		\draw [fill=black] (10.94,4.42) circle (1.5pt);
		\draw [fill=black] (7.3,4.42) circle (1.5pt);
		\end{scriptsize}
		\end{tikzpicture}
		\caption{The graph $G'$ which is not a simplicial subgraph of $G$}
		\label{G''}
	\end{center}
\end{figure}
\end{ex}

\begin{rem}
One may easily check that a simplicial subgraph of a chordal graph is again a chordal graph. However, it is not known that whether any simplicial subclutter of a chordal clutter is again chordal.
\end{rem}

\subsection{Multigraded minimal free resolutions of simplicial subcluttres}
Let $\mathcal{D}$ be a simplicial subclutter of a $d$-uniform clutter $\mathcal{C}$. It follows from \cite[Theorem~2.1]{BYZ} that the ideals $I (\bar{\C})$ and $I (\bar{\mathcal{D}})$ share the same non-linear strands (i.e. $\beta_{i,j} \left( I (\bar{\C}) \right) = \beta_{i,j} \left( I (\bar{\mathcal{D}}) \right)$, for all $i,j$ with $j-i>d$). Theorem~\ref{inequality of multigraded betti numbers-thm} describes all the Betti numbers of $I(\bar{\mathcal{D}})$ with respect to those of $I(\bar{\mathcal{C}})$. Before stating this theorem, we first   prove the following proposition which is a crucial point for the proof of the desired conclusion.

\begin{prop}[{compare to \cite[Proposition 4.2]{MS}}] \label{preliminary statement}
Let $\mathcal{H}$ be a $d$-uniform clutter on the vertex set $[n]$, $e \in \mathrm{Simp} (\mathcal{H})$ and $\Delta = \Delta \left( \mathcal{H} \right)$  the clique complex of $\mathcal{H}$.
\begin{itemize}
\item[\rm (a)] If  $\Delta^\star = \Delta \setminus \{ G \in \Delta \colon \, e \subseteq G \}$, then 
$$\widetilde{H}_i \left( \Delta; \mathbb{K} \right) \cong \widetilde{H}_i \left( \Delta^\star; \mathbb{K} \right), \text{ for all }i > d-2.$$

\item[\rm (b)] Let $\Delta'$ be a subcomplex of $\Delta$ defined as follows: $\Delta'= \Delta \setminus \{ G \in \Delta \colon \, F \subseteq G \}$, where $F=e \cup \{v\}$ for some $v\in N_\mathcal{H}[e]\setminus e$.   Set $\Delta'=\Delta$, if there is no such $v$. Then 
%We define the simplicial complex $\Delta'$ as follows: If $N_{\mathcal{H}}[e] \setminus e = \varnothing$, we set $\Delta'= \Delta$. Otherwise, pick an element $v \in N_{\mathcal{H}}[e] \setminus e$ and let $\Delta' = \Delta \setminus \{ G \in \Delta \colon \, F \subseteq G \}$, where $F=e \cup \{v\}$. Then 
$$\widetilde{H}_i \left( \Delta; \mathbb{K} \right) \cong \widetilde{H}_i \left( \Delta'; \mathbb{K} \right), \text{ for all }i \neq d-2.$$

\item[\rm (c)] the following sequence is exact:
\begin{equation*}
0 \to \widetilde{H}_{d-2} \left( \Delta^\star; \mathbb{K} \right) \to \widetilde{H}_{d-2} \left( \Delta; \mathbb{K} \right) \to \widetilde{H}_{d-3} \left( \partial e * \langle N_{\mathcal{H}}[e] \setminus e \rangle; \mathbb{K} \right) \to \widetilde{H}_{d-3} \left( \Delta^\star; \mathbb{K} \right) \to 0.
\end{equation*}
\end{itemize}
\end{prop}

\begin{proof}
(a) Note that
\begin{itemize}
\item[(i)] $\Delta = \Delta^\star \cup \langle N_{\mathcal{H}}[e] \rangle$, and
\item[(ii)]  $\Delta^\star \cap \langle N_{\mathcal{H}}[e] \rangle= \partial e * \langle N_{\mathcal{H}}[e] \setminus e \rangle$, where $\partial e$ is the pure $(d-3)$-skeleton of $\langle e\rangle$.
\end{itemize}
It follows from (ii) that for all $i  \neq d-3$
\begin{equation} \label{internal equation 3}
\widetilde{H}_i \left( \Delta^\star \cap \langle N_{\mathcal{H}}[e] \rangle; \mathbb{K} \right) = 0.
\end{equation}
Indeed, if $N_{\mathcal{H}}[e] \setminus e = \varnothing$, then $\Delta^\star \cap \langle N_{\mathcal{H}}[e] \setminus e \rangle =  \partial e$ which is a pure $(d-3)$-skeleton of of the simplex $\langle e \rangle$. Otherwise, $\Delta^\star \cap \langle N_{\mathcal{H}}[e]  \rangle$ is the join of $\partial e$ with a simplex. Hence all of its homologies vanish.

Using (\ref{internal equation 3}) and the reduced  Mayer-Vietoris long exact sequence 
\begin{align*}
\cdots \to \widetilde{H}_i \left( \Delta^\star \cap \langle N_{\mathcal{H}}[e] \rangle; \mathbb{K} \right) \to \widetilde{H}_i \left( \Delta^\star; \mathbb{K} \right) \oplus \widetilde{H}_i \left( \langle N_{\mathcal{H}}[e]  \rangle; \mathbb{K} \right) \to \widetilde{H}_i \left( \Delta; \mathbb{K} \right) \\ 
\to \widetilde{H}_{i-1} \left( \Delta^\star \cap \langle N_{\mathcal{H}}[e] \rangle; \mathbb{K} \right) \to \cdots
\end{align*}
we conclude that for all $i> d-2$, 
\begin{equation*}
\widetilde{H}_i \left( \Delta; \mathbb{K} \right) \cong  \widetilde{H}_i \left( \Delta^\star; \mathbb{K} \right) \oplus \widetilde{H}_i \left( \langle N_{\mathcal{H}}[e]  \rangle; \mathbb{K} \right).
\end{equation*}
Since $\langle N_{\mathcal{H}}[e] \rangle$ is a simplex, $\widetilde{H}_i \left( \langle N_{\mathcal{H}}[e] \rangle; \mathbb{K} \right) =0$, for all $i$. Therefore (a) is obtained.

\medskip
(b) If $\Delta'=\Delta$, then we are done. Assume that $N_{\mathcal{H}}[e] \setminus e \neq \varnothing$ and $v$ and $F$ are as in statement (b) above. One may easily check that $\Delta'=\Delta(\mathcal{H}\setminus \{F\})$. Hence $\widetilde{H}_i \left( \Delta'; \mathbb{K} \right)=\widetilde{H}_i \left( \Delta; \mathbb{K} \right)=0$ for all $i<d-2$ (\cite[Proposition 3.1]{MAR}). On the other hand, we have:
\begin{itemize}
\item[(iii)] $\Delta' = \Delta^\star \cup \langle N_{\mathcal{H}}[e] \setminus \{v \} \rangle$, and
\item[(iv)]  $\Delta^\star \cap \langle N_{\mathcal{H}}[e] \setminus \{v\} \rangle= \partial e * \langle N_{\mathcal{H}}[e] \setminus F \rangle$.
\end{itemize}
One may use the same method as in (a) to obtain $\widetilde{H}_i \left( \Delta'; \mathbb{K} \right) \cong \widetilde{H}_i \left( \Delta^\star; \mathbb{K} \right)$, for all $i > d-2$. Now the result follows from (a).

\medskip
(c) Consider the following Mayer-Vietoris sequence:
\begin{equation} \label{internal equation 9}
\begin{split}
\cdots &\to \widetilde{H}_{d-2} \left( \partial e * \langle N_{\mathcal{H}}[e] \setminus e \rangle ; \mathbb{K}\right) \\
& \to \widetilde{H}_{d-2} \left( \Delta^\star; \mathbb{K} \right) \oplus \widetilde{H}_{d-2} \left( \langle N_{\mathcal{H}}[e]\rangle; \mathbb{K} \right) \to \widetilde{H}_{d-2} \left( \Delta; \mathbb{K} \right) \\ 
& \to \widetilde{H}_{d-3} \left( \partial e * \langle N_{\mathcal{H}}[e] \setminus e \rangle; \mathbb{K} \right) \to \widetilde{H}_{d-3} \left( \Delta^\star; \mathbb{K} \right) \oplus \widetilde{H}_{d-3} \left( \langle N_{\mathcal{H}}[e]\rangle; \mathbb{K} \right) \\
& \to \widetilde{H}_{d-3} \left( \Delta; \mathbb{K} \right) \to \cdots.
\end{split}
\end{equation}
One should note that
\begin{itemize}
\item $\widetilde{H}_{d-2} \left( \langle N_{\mathcal{H}}[e]\rangle; \mathbb{K} \right) = \widetilde{H}_{d-3} \left( \langle N_{\mathcal{H}}[e]\rangle; \mathbb{K} \right) = 0$, since $\langle N_{\mathcal{H}}[e]\rangle$ is a simplex.
\item $\widetilde{H}_{d-3} \left( \Delta; \mathbb{K} \right) =0$, since %$I_\Delta = I ( \bar{\mathcal{H}})$ is either zero or generated in degree $d$.
$\langle [n]\rangle^{[k]}$ is a subcomplex of $\Delta$ for all $k\leq d-2$. 
\item $\widetilde{H}_{d-2} \left( \partial e * \langle N_{\mathcal{H}}[e] \setminus e \rangle; \mathbb{K} \right) =0$, because
\begin{itemize}
\item If $\langle N_{\mathcal{H}}[e] \setminus e \rangle = \varnothing$, then $\partial e * \langle N_{\mathcal{H}}[e] \setminus e \rangle = \partial e$ is a simplicial complex of dimension $d-3$.
\item If $\langle N_{\mathcal{H}}[e] \setminus e \rangle \neq \varnothing$, then $\partial e * \langle N_{\mathcal{H}}[e] \setminus e \rangle$ is the join of $\partial e$ and a simplex.
\end{itemize}
\end{itemize}
By applying the above facts to (\ref{internal equation 9}), we obtain the long exact sequence as in (c).
%\begin{equation*} 
%0 \to \widetilde{H}_{d-2} \left( \Delta^\star; \mathbb{K} \right) \to \widetilde{H}_{d-2} \left( \Delta; \mathbb{K} \right) \to \widetilde{H}_{d-3} \left( \partial e * \langle N_{\mathcal{H}}[e] \setminus e \rangle; \mathbb{K} \right) \to \widetilde{H}_{d-3} \left( \Delta^\star; \mathbb{K} \right) \to 0.
%\end{equation*}
\end{proof}

Now we are ready to prove the main theorem of this section.
\begin{thm} \label{inequality of multigraded betti numbers-thm}
Let $I=I(\bar{\C})$, where $\C$ is a $d$-uniform clutter on $[n]$. Let  $e \in \mathrm{Simp}(\C)$ and let $u:=\mathbf{x}_F=\prod_{i\in F}x_i$ with $F\in \C$ and $e\subset F$.  Then for all $i \geq 0$ and  all $\textbf{{\rm\textbf{a}}} \in \mathbb{Z}^n$ one has:
\begin{equation} \label{inequality of multigraded betti numbers}
\beta_{i, \textbf{{\rm\textbf{a}}}} \left( I+(u)  \right) =  \beta_{i, \textbf{{\rm\textbf{a}}}} \left( I \right) + 
\begin{cases}
1, & \text{if } |\textbf{{\rm\textbf{a}}}| =d+i, \; F \subseteq \mathrm{supp} (\textbf{{\rm\textbf{a}}}) \text{ and } \left( x_j \colon \; j \in \mathrm{supp} (\textbf{{\rm\textbf{a}}}) \setminus F \right) \subset \left(I \colon \textbf{\rm \textbf{x}}_F \right)\\
0, & \text{otherwise}
\end{cases}
\end{equation}
\end{thm}

\begin{proof}
Let  $F=e \cup \{v\}$ and $\mathcal{D} = \mathcal{C} \setminus \{F\}$. For $W \subseteq [n]$, set $\mathcal{H}=\C_W$ and let $\Delta$, $\Delta'$ and $\Delta^{\star}$ be as in Proposition~\ref{preliminary statement}. It is easy to see that $\Delta(\C_W) = \Delta(\C)_W$, $\Delta(\C_W)' = \Delta(\mathcal{D})_W$ and $\Delta(\C_W)^\star=(\Delta(\C)^\star)_W$. 

If $e \not\subset W$, then $\Delta (\C)_W = \Delta (\mathcal{D})_W$ and $\Delta(\C_W)=\Delta(\C_W)^\star$. Otherwise, $e \in \mathrm{Simp}(\C_W)$ and by Proposition~\ref{preliminary statement} we conclude that
\begin{equation} \label{internal equation 7}
\widetilde{H}_i \left( \Delta(\C)_W; \mathbb{K} \right) \cong \widetilde{H}_i \left( \Delta(\mathcal{D})_W; \mathbb{K} \right), \quad \text{for all } i>d-2.
\end{equation}
Moreover, by using Proposition~\ref{preliminary statement}(c), we obtain
\begin{equation} \label{internal equation 13}
\begin{split}
\dim_\mathbb{K}\widetilde{H}_{d-2} \left( \Delta(\C)_W; \mathbb{K} \right) = \dim_\mathbb{K}\widetilde{H}_{d-2} \left((\Delta(\C)^\star)_W ; \mathbb{K} \right) -\dim_\mathbb{K}\widetilde{H}_{d-3} \left( (\Delta(\C)^\star)_W; \mathbb{K} \right) \\
 + \dim_\mathbb{K}\widetilde{H}_{d-3} \left( \partial e * \langle N_{\mathcal{\C}_W}[e] \setminus e \rangle ; \mathbb{K}\right).
\end{split}
\end{equation}

Assume that $\textbf{{\rm\textbf{a}}} \in \mathbb{Z}^n$ with $W=\mathrm{supp} (\textbf{{\rm\textbf{a}}})$. First suppose $|\textbf{{\rm\textbf{a}}}| >d+i$. Using a theorem of Hochster (see e.g. \cite[Theorem 8.1.1]{HHBook}) we have
\begin{align*}
\beta_{i, \textbf{{\rm\textbf{a}}}} \left( I (\bar{\C}) \right) 
 = \beta_{i, \textbf{{\rm\textbf{a}}}} \left( I_{\Delta(\mathcal{C})} \right) 
& = \dim_\mathbb{K} \mathrm{Tor}_i (\mathbb{K}; I_{\Delta(\mathcal{C})})_{\textbf{{\rm\textbf{a}}}} \\
& = \dim_\mathbb{K} \widetilde{H}_{|W|-i-2} \left( \Delta(\C)_W; \mathbb{K} \right) \\
& = \dim_\mathbb{K} \widetilde{H}_{|W|-i-2} \left( \Delta(\mathcal{D})_W; \mathbb{K} \right) \tag{\text{by }(\ref{internal equation 7})}\\
& = \dim_\mathbb{K} \mathrm{Tor}_i (\mathbb{K}; I_{\Delta(\mathcal{D})})_{\textbf{{\rm\textbf{a}}}} \\
& = \beta_{i, \textbf{{\rm\textbf{a}}}} \left( I_{\Delta(\mathcal{D})} \right) = \beta_{i, \textbf{{\rm\textbf{a}}}} \left( I (\bar{\mathcal{D}}) \right).
\end{align*}

Now suppose that $|\textbf{{\rm\textbf{a}}}| =d+i$.  If $e\not\subset \mathrm{supp}(\textbf{{\rm\textbf{a}}})$, then as mentioned before we have $\Delta(\mathcal{C})_W=\Delta(\mathcal{D})_W$. The same discussion as above shows that $\beta_{i, \textbf{{\rm\textbf{a}}}} \left( I (\bar{\C}) \right) =\beta_{i, \textbf{{\rm\textbf{a}}}} \left( I (\bar{\mathcal{D}}) \right)$. Suppose $e\subset \mathrm{supp}(\textbf{{\rm\textbf{a}}})$.  
Using the theorem of Hochster and (\ref{internal equation 13}), we have
\begin{align*}
\beta_{i, \textbf{{\rm\textbf{a}}}} \left( I (\bar{\C}) \right) 
& = \dim_\mathbb{K} \widetilde{H}_{d-2} \left( \Delta(\C)_W; \mathbb{K} \right)  \\
& = \dim_\mathbb{K}\widetilde{H}_{d-2} \left( (\Delta(\C)^\star)_W; \mathbb{K} \right) -\dim_\mathbb{K}\widetilde{H}_{d-3} \left( (\Delta(\C)^\star)_W; \mathbb{K} \right) \\
& \quad + \dim_\mathbb{K}\widetilde{H}_{d-3} \left( \partial e * \langle N_{\mathcal{\C}_W}[e] \setminus e \rangle ;\mathbb{K} \right)\\
& = \beta_{i, \textbf{{\rm\textbf{a}}}} \left( I_{\Delta(\C)^\star} \right) - \beta_{i+1, \textbf{{\rm\textbf{a}}}} \left( I_{\Delta(\C)^\star} \right) + \dim_\mathbb{K}\widetilde{H}_{d-3} \left( \partial e * \langle N_{\mathcal{\C}_W}[e] \setminus e \rangle ; \mathbb{K} \right).
\end{align*}
Hence for $\textbf{{\rm\textbf{a}}} \in \mathbb{Z}^n$ with $|\textbf{{\rm\textbf{a}}}| =d+i$ and $e\subset \mathrm{supp}(\textbf{{\rm\textbf{a}}})$, and for $W=\mathrm{supp} (\textbf{{\rm\textbf{a}}})$, one has
\begin{equation*}
\beta_{i, \textbf{{\rm\textbf{a}}}} \left( I (\bar{\C}) \right) = %\begin{cases}
%\beta_{i, \textbf{{\rm\textbf{a}}}} \left( I_{\Delta(\C)^\star} \right), & \text{if } e \not\subset \mathrm{supp} (\textbf{{\rm\textbf{a}}})\\
\beta_{i, \textbf{{\rm\textbf{a}}}} \left( I_{\Delta(\C)^\star} \right) - \beta_{i+1, \textbf{{\rm\textbf{a}}}} \left( I_{\Delta(\C)^\star} \right) + \dim_\mathbb{K}\widetilde{H}_{d-3} \left( \partial e * \langle N_{\mathcal{\C}_W}[e] \setminus e \rangle ; \mathbb{K} \right).%, &\text{Otherwise}
%\end{cases}
\end{equation*}

\medskip
Now set $\mathcal{H}=\mathcal{D}_W$ in Proposition~\ref{preliminary statement} and let $\Delta$ and $\Delta^\star$ be as defined in that proposition. Note that if $e\subset W$, then $e\in \mathrm{simp}(\mathcal{D}_W)$. It is seen that $\Delta(\mathcal{D}_W)^\star=\left(\Delta(\mathcal{C})^\star\right)_W$, because $\mathcal{D}\setminus e=\mathcal{C}\setminus e$.  
A similar argument as above yields the following:
\begin{equation*}
\beta_{i, \textbf{{\rm\textbf{a}}}} \left( I (\bar{\mathcal{D}}) \right) = %\begin{cases}
%\beta_{i, \textbf{{\rm\textbf{a}}}} \left( I_{\Delta(\C)^\star} \right), & \text{if } e \not\subset \mathrm{supp} (\textbf{{\rm\textbf{a}}})\\
\beta_{i, \textbf{{\rm\textbf{a}}}} \left( I_{\Delta(\C)^\star} \right) - \beta_{i+1, \textbf{{\rm\textbf{a}}}} \left( I_{\Delta(\C)^\star} \right) + \dim_\mathbb{K}\widetilde{H}_{d-3} \left( \partial e * \langle N_{\mathcal{D}_W}[e] \setminus e \rangle ; \mathbb{K} \right),%, &\text{Otherwise.}
%\end{cases}
\end{equation*}
where $|\textbf{{\rm\textbf{a}}}| =d+i$, $e\subset \mathrm{supp}(\textbf{{\rm\textbf{a}}})$, and  $W=\mathrm{supp} (\textbf{{\rm\textbf{a}}})$.

Since  $N_{\mathcal{D}_W}[e] \setminus e=N_{\mathcal{C}_W}[e] \setminus F$,  for $\textbf{{\rm\textbf{a}}} \in \mathbb{Z}^n$ with $|\textbf{{\rm \textbf{a}}}| =d+i$, we obtain 
\begin{align*}
\beta_{i, \textbf{{\rm\textbf{a}}}} & \left( I (\bar{\mathcal{D}}) \right)  =  
\beta_{i, \textbf{{\rm\textbf{a}}}} \left( I(\bar{\mathcal{C}}) \right) + \\
& \begin{cases}
0, & \text{if } e \not\subset W\\
\dim_\mathbb{K}\widetilde{H}_{d-3} \left( \partial e * \langle N_{\mathcal{\C}_W}[e] \setminus F \rangle ; \mathbb{K} \right) - \dim_\mathbb{K}\widetilde{H}_{d-3} \left( \partial e * \langle N_{\mathcal{\C}_W}[e] \setminus e \rangle ; \mathbb{K} \right) & \text{otherwise}.
\end{cases}
\end{align*}
Note that
\begin{align*} 
\dim_\mathbb{K}\widetilde{H}_{d-3} \left( \partial e * \langle N_{\mathcal{\C}_W}[e] \setminus F \rangle ; \mathbb{K} \right) - \dim_\mathbb{K}\widetilde{H}_{d-3} \left( \partial e * \langle N_{\mathcal{\C}_W}[e] \setminus e \rangle ; \mathbb{K} \right)=
\begin{cases}
1, &\text{if \ }N_{\mathcal{\C}_W}[e] =F\\
0, &\text{otherwise.}
\end{cases}
\end{align*}
Furthermore, $N_{\mathcal{\C}_W}[e] =F$ is equivalent to
$$F \subseteq \mathrm{supp} (\textbf{{\rm\textbf{a}}}) \text{ and } \left( x_j \colon \; j \in \mathrm{supp} (\textbf{{\rm\textbf{a}}}) \setminus F \right) \subset \left(I \colon \textbf{\rm \textbf{x}}_F \right).$$
This completes the proof.
\end{proof}
%%%%%%%%%%%%%%%%%%%%% Splitting proof
%The Betti formula (\ref{inequality of multigraded betti numbers}) rings a bell whether the ideal $J$ defined as in Theorem~\ref{inequality of multigraded betti numbers-thm} is splittable. 
\medskip
Though Theorem~\ref{inequality of multigraded betti numbers-thm} gives a direct proof for formula (\ref{inequality of multigraded betti numbers}), one can obtain it showing that $J$ is Betti splittable. Indeed, the following  discussion shows that $J=I+(\mathbf{x}_F)$ is even splittable.

Let $I$ and $J$ and $F$ be as defined in Theroem~~\ref{inequality of multigraded betti numbers-thm} and let $K:=(\mathbf{x}_F)$. Note that by \cite[Propositions~1.2.1, 1.2.2]{HHBook}, we have 
\begin{eqnarray*}
I\cap K=(\mathbf{x}_F)(I:\mathbf{x}_F)%=(\mathbf{x}_F)(\frac{u}{\gcd(u,\mathbf{x}_F)}:\ u\in \mathcal{G}(I)).
\end{eqnarray*}
It follows from Lemma~\ref{simplicial subclutter has linear quotients}(b) that
\begin{eqnarray*}\label{intersection}
	I\cap K=(\mathbf{x}_F)(x_i:\ x_i\mathbf{x}_e\in I).
\end{eqnarray*}

Now define  $\phi: \mathcal{G}(I\cap K)\to \mathcal{G}(I)$ and $\psi: \mathcal{G}(I\cap K)\to \mathcal{G}(K)$ to be the functions with $\phi(x_i\mathbf{x}_F)=x_i\mathbf{x}_e$ and $\psi(x_i\mathbf{x}_F)=\mathbf{x}_F$.
%It follows that $\lcm(\phi(x_i\mathbf{x}_F),\psi(x_i\mathbf{x}_F))=x_i\mathbf{x}_F$. Moreover, 
%setting $S=\{x_{i_1}\mathbf{x}_F,\ldots, x_{i_k}\mathbf{x}_F\}\subset \mathcal{G}(I\cap K)$, we have  $\lcm(\phi(S))=\mathbf{x}_e\prod_{j=1}^kx_{i_j}$ and $\lcm(\psi(S))=\mathbf{x}_F$ and both strictly divide $\lcm(S)=\mathbf{x}_F\prod_{j=1}^kx_{i_j}$. Hence 
It is easy to check that the function $\mathcal{G}(I\cap K)\to \mathcal{G}(I)\times \mathcal{G}(K)$ with $w\mapsto (\phi(w), \psi(w))$ is a splitting function. Hence $J=I+(\mathbf{x}_F)$ is a splitting and so $J$ is Betti splittable. That is for ${\rm \textbf{a}}\in\mathbb{Z}^n$, 
\begin{equation}\label{multidegree split}
\beta_{i,{\rm \textbf{a}}}(J) = \beta_{i,,{\rm \textbf{a}}}(I) + \beta_{i,,{\rm \textbf{a}}}((\mathbf{x}_F)) + \beta_{i-1,,{\rm \textbf{a}}}(I \cap (\mathbf{x}_F)).
\end{equation}
Note that
\begin{align*}
\beta_{i,{\rm \textbf{a}}}((\mathbf{x}_F))=
\begin{cases}
1,& \text{if } i=0,\ \mathrm{supp}({\rm \textbf{a}})=F\\
0,&\text{otherwise.}
\end{cases}
\end{align*}
Since by Lemma~\ref{simplicial subclutter has linear quotients}(b), the ideal $I:\mathbf{x}_F$ is generated by variables, setting $\deg x_i=e_i$, where $e_i$ is the $i$-th standard basis vector, the Koszul complex is a multigarded minimal free resolution of $I:\mathbf{x}_F$. Therefore, 
%\begin{align*}
%\beta_{i-1, {\rm \textbf{a}}}(I:\mathbf{x}_F)=
%\begin{cases}
%	1& \text{if } {\rm \textbf{a}}=\deg \prod_{k=1}^{i}x_{j_k} \text{with } x_{j_k}\in I:\mathbf{x}_F\\
%	0& \text{Otherwise.}
%\end{cases}
%\end{align*}
\begin{align*}
\beta_{i-1, {\rm \textbf{a}}}(I:\mathbf{x}_F)=
\begin{cases}
1,& \text{if } \supp({\rm \textbf{a}})=\{j_1,\ldots, j_i\} \text{with } x_{j_k}\in I:\mathbf{x}_F\\
0,& \text{otherwise.}
\end{cases}
\end{align*}
It follows that 
%\begin{align*}
%\beta_{i-1, {\rm \textbf{a}}}(I\cap (\mathbf{x}_F))=
%\begin{cases}
%1& \text{if } {\rm \textbf{a}}=\deg \mathbf{x}_F+\deg \prod_{k=1}^{i}x_{j_k} \text{with } x_{j_k}\in I:\mathbf{x}_F\\
%0& \text{Otherwise.}
%\end{cases}
%\end{align*}
\begin{align*}
\beta_{i-1, {\rm \textbf{a}}}(I\cap (\mathbf{x}_F))=
\begin{cases}
1,& \text{if } \supp({\rm \textbf{a}})=F\cup \{j_1,\ldots, j_i\}  \text{ with } x_{j_k}\in I:\mathbf{x}_F\\
0,& \text{otherwise.}
\end{cases}
\end{align*}

Applying the above formulas in (\ref{multidegree split}) we get (\ref{inequality of multigraded betti numbers}).

%%%%%%%%%%%%%%%%%%%%%%%%%%%%%%%%%%%%%%%%%%%%%%

\begin{cor} \label{chordallin}
Let $\mathcal{C}$ be a $d$-uniform clutter on the vertex set $[n]$ and let $\mathcal{D}$ be a simplicial subclutter of $\C$. Then
\begin{itemize}
\item[\rm (a)] {\rm(}\cite[Theorem 2.1]{BYZ}{\rm)} $\beta_{i,j} \left( I (\bar{\C}) \right) = \beta_{i,j} \left( I (\bar{\mathcal{D}}) \right)$, for all $i,j$ with $j-i>d$.
\item[\rm (b)] If $\C \neq \C_{n,d}$, then $\mathrm{reg} \left( I (\bar{\C}) \right) = \mathrm{reg} \left( I (\bar{\mathcal{D}}) \right)$.
\item[\rm (c)] $I (\bar{\C})$ has a linear resolution if and only if $I (\bar{\mathcal{D}})$ has a linear resolution.
\item[\rm (d)] {\rm(}\cite[Theorem 3.3]{BYZ}{\rm)} If $\C \neq \C_{n,d}$ is a chordal clutter, then $I(\bar{\mathcal{C}})$ has a $d$-linear resolution over all fields.
\item[\rm(e)] there exists $u\notin I(\bar{\mathcal{C}})$ such that the non-linear Betti numbers of $I(\bar{\mathcal{C}})$ and $I(\bar{\mathcal{C}})+(u)$ are the same.
\end{itemize}
\end{cor}

\begin{proof}
One may note that
$$\beta_{i,j} (I) = \sum\limits_{\textbf{{\rm\textbf{a}}} \in \mathbb{Z}^n \atop |\textbf{{\rm\textbf{a}}}| =j} \beta_{i, \textbf{{\rm\textbf{a}}}}  (I).$$
So that (a) follows directly from Theorem~\ref{inequality of multigraded betti numbers-thm}. The statements in (b) and (c) are direct consequences of (a). Note that if $\C$ is a chordal clutter, then $\varnothing$ is a simplicial subclutter of $\C$. Since the ideal $I(\bar{\varnothing}) = I ({\C_{n,d}})$ has a $d$-linear resolution over all fields,  we conclude from (c) that
%$\beta_{i,j} \left( I (\bar{\C}) \right) = \beta_{i,j} \left( I (\bar{\varnothing}) \right) = \beta_{i,j} \left( I ({\C_{n,d}}) \right) =0$, for $j-i>d$. This implies that 
the ideal $I(\bar{\mathcal{C}})$ has a $d$-linear resolution over all fields.
\end{proof}

\begin{cor}\label{graded linear strands}
Let $\C$ be a $d$-uniform clutter on the vertex set $[n]$ and $e \in \mathrm{Simp}(\C)$.  Suppose that $A \subseteq \{F \in \C \colon \; e \subset F\}$ be a non-empty set. Let $\mathcal{D} = \C \setminus A$, $I=I(\bar{\C})$ and $J= I(\bar{\mathcal{D}} )= I+\left( \textbf{\rm \textbf{x}}_F \colon \;F \in A \right)$. Then
\begin{equation*}
\beta_{i, i+d} (J) = \beta_{i, i+d} (I) + \sum\limits_{j=0}^{|A|-1}{s+j \choose i},
\end{equation*}
where $s$ is the number of  minimal generators of the ideal $I \colon  \textbf{\rm \textbf{x}}_F$ for an arbitrary $F \in A$.
\end{cor}

\begin{proof}
Let $F \in A$. By virtue of Theorem~\ref{inequality of multigraded betti numbers-thm}, we have
\begin{align*}
\beta_{i,i+d} (I+\left( \textbf{\rm \textbf{x}}_F \right)) & = \sum\limits_{\textbf{{\rm\textbf{a}}} \in \mathbb{Z}^n \atop |\textbf{{\rm\textbf{a}}}|=i+d} \beta_{i, \textbf{{\rm\textbf{a}}}} (I+\left( \textbf{\rm \textbf{x}}_F \right)) \\
& = \beta_{i, i+d} (I) + \sum\limits_{\textbf{{\rm\textbf{a}}}} 1,
\end{align*}
where the sum is taken over all $\textbf{{\rm\textbf{a}}} \in \mathbb{Z}^n$ with $|\mathrm{supp}(\textbf{{\rm\textbf{a}}})|=i+d$, $F \subseteq \mathrm{supp} (\textbf{{\rm\textbf{a}}})$ and 
$$\left( x_j \colon \; j \in \mathrm{supp} (\textbf{{\rm\textbf{a}}}) \setminus F \right) \subset \left(I \colon \textbf{\rm \textbf{x}}_F \right).$$
Note that the ideal $I \colon \textbf{\rm \textbf{x}}_{F}$ is generated by a subset of variables, by Lemma~\ref{simplicial subclutter has linear quotients}(b). Counting the number of $\textbf{{\rm\textbf{a}}} \in \mathbb{Z}^n$ with the above properties, we get
\begin{equation*}
\beta_{i,i+d} (I+\left( \textbf{\rm \textbf{x}}_F \right)) = \beta_{i, i+d} (I) + {s \choose i},
\end{equation*}
where  $s$ is the number of  minimal generators of the ideal $I \colon  \textbf{\rm \textbf{x}}_F$. This proves the assertion in the case that $|A|=1$. Let now $A=\{F_1, \ldots F_l\}$ with $l>1$. Then $J=J'+ \left( \textbf{\rm \textbf{x}}_{F_l} \right)$, where $J'= I+\left( \textbf{\rm \textbf{x}}_{F_1}, \ldots, \textbf{\rm \textbf{x}}_{F_{l-1}} \right)$. By induction hypothesis on $l$, we conclude that
\begin{align*}
\beta_{i, i+d} (J) &= \beta_{i, i+d} (J') + {s_{F_l} \choose i} \\
&= \beta_{i, i+d}(I)+  \left(\sum\limits_{j=0}^{|A|-2} {s+j \choose i} \right)+ {s_{F_l} \choose i},
\end{align*}
where $s_{F_l}$ is the number of  minimal generators of $J'\colon \textbf{\rm \textbf{x}}_{F_l}$ and $s$ is the number of  minimal generators of $I \colon \textbf{\rm \textbf{x}}_{F}$ for an arbitrary element $F \in A \setminus \{F_l\}$. Note that $s_{F_l} = (l-1) + s=(|A|-1) + s$, by Lemma~\ref{simplicial subclutter has linear quotients}(b). This implies  the desired result.
\end{proof}

\begin{cor}\label{graded linear strands-2}
Let $\mathcal{D}$ be a simplicial subclutter of $\C$ obtained by the simplicial sequence $\textbf{\rm \textbf{e}}= e_1, \ldots, e_r$ and the sets $A_1, \ldots, A_r$, as defined in Definition~\ref{Definition of simplicial subclutter}. Let $I=I(\bar{\C})$ and $J=I(\bar{\mathcal{D}})$. Then
\begin{equation} \label{graded linear stand-formula}
\beta_{i, i+d} (J) = \beta_{i, i+d} (I) + \sum\limits_{k=1}^{r} \sum\limits_{j=0}^{|A_k|-1} {t_k+s_k+j \choose i},
\end{equation}
where 
\begin{itemize}
\item for $1 \leq k \leq r$, $s_k$ is the number of  minimal generators of the ideal $I \colon  \textbf{\rm \textbf{x}}_{F}$, where $F$ is an arbitrary element in $A_k$.
\item $t_1=0$, $t_k = |\{ F \in \mathop{\bigcup}\limits_{j=1}^{k-1} A_j \colon \; e_k \subset F \}|$, for $2 \leq k\leq r$.
\end{itemize}
In particular, 
$$\mathrm{pd}(S/J) = \mathop{\max}\limits_{1 \leq k\leq r} \{ \mathrm{pd}(S/I), t_k+s_k+|A_k| \}.$$
\end{cor}

\begin{proof}
We have $J=I+\sum\limits_{k=1}^{r}\left( \textbf{\rm \textbf{x}}_F \colon F \in A_k \right)$. We use induction on $r$ to prove the assertion. For $r=1$,  Corollary~\ref{graded linear strands} gives the desired formula. Suppose  $r>1$.  Note that we have $J = J' + \left( \textbf{\rm \textbf{x}}_F \colon F \in A_r \right)$, where
$$J'= I + \sum\limits_{k=1}^{r-1} \left( \textbf{\rm \textbf{x}}_F \colon F \in A_k \right).$$
By Corollary~\ref{graded linear strands} and induction hypothesis we get
\begin{align*}
\beta_{i, i+d} (J) &= \beta_{i, i+d} (J') + \sum_{j=0}^{|A_r|-1} {s'_r+j \choose i} \\
&= \beta_{i, i+d}(I)+ \sum\limits_{k=1}^{r-1} \sum\limits_{j=0}^{|A_k|-1} {t_k+s_k+j \choose i} + \sum_{j=0}^{|A_r|-1} {s'_r+j \choose i},
\end{align*}
where $s_k$ and $t_k$ are as explained in the statement and $s'_r$ is the number of  minimal generators of $J' \colon \textbf{\rm \textbf{x}}_{F}$ for an arbitrary element $F \in A_r$. Note that $s'_r = t_r + s_r$ by Lemma~\ref{simplicial subclutter has linear quotients}(b). This completes the proof of the first part.

\medspace
Now we prove the second part of the statement. %let $\rho_J$ (resp. $\rho_I$) denotes the projective dimension of $S/J$ (resp. $S/I$). Then a 
A direct consequence of (\ref{graded linear stand-formula}) and Corollary~\ref{chordallin}(a), yields that $\mathrm{pd}(S/I)  \leq \mathrm{pd}(S/J)$. Moreover, it follows from (\ref{graded linear stand-formula}) that $s_k+t_k+|A_k| \leq \mathrm{pd}(S/J) $, for all $1 \leq k \leq r$. Hence $\mathrm{pd}(S/J)  \leq {\max}_{1 \leq k\leq r} \{ \mathrm{pd}(S/I), t_k+s_k+|A_k| \}$. For the converse inequality, assume that $\mathrm{pd}(S/I)  < \mathrm{pd}(S/J)$. Then, by virtue of Corollary~\ref{chordallin}(a) we conclude that, $\beta_{\mathrm{pd}(S/J), \mathrm{pd}(S/J) +d} (S/J) \neq 0$. Note that $\beta_{\mathrm{pd}(S/J), \mathrm{pd}(S/J)+d} (S/I) = 0$, because $\mathrm{pd}(S/I)  < \mathrm{pd}(S/J)$. However, by (\ref{graded linear stand-formula}) we know that $\beta_{\mathrm{pd}(S/J), \mathrm{pd}(S/J)+d} (S/J) \neq 0$ if and only if there exists $1 \leq k \leq r$ such that $s_k+t_k+|A_k| \geq \mathrm{pd}(S/J)$. This completes the proof.
\end{proof}

In Corollary~\ref{chordallin}(c) we saw that for a simplicial subclutter $\mathcal{D}$ of a $d$-uniform clutter $\C$, 
%\begin{quote}
$I(\bar{\C})$ has a linear resolution if and only if so does $I(\bar{\mathcal{D}})$.
%\end{quote}
One may ask if this statement holds  when we replace ``linear resolution" with ``linear quotients". 
In the following we show that the ``only if" direction is true, while the other one is not. First we need the following lemma.

\begin{lem}  \label{simplicial subclutter has linear quotients}
Let $\C$ be a $d$-uniform clutter on the vertex set $[n]$ and let $I= I (\bar{\C})$. Let $e$ be a $(d-1)$-subset of $[n]$. Then
\begin{itemize}
\item[\rm (a)] $e \in \mathrm{Simp}(\C)$ if and only if for all $v \in I$, there exists $i \in  [n]\setminus N_{\C}[e]$ such that $x_i|v$.
\item[\rm (b)] {\rm(\cite{Big})} If $e \in \mathrm{Simp}(\C)$ and $F$ is a circuit of $\C$ which contains $e$, then $I \colon \textbf{\rm \textbf{x}}_{F} = \left(x_i \colon \, x_i \textbf{\rm \textbf{x}}_{e} \in I \right)$.
\item[\rm (c)] If $I$ has linear quotients with respect to $\xb_{F_1}, \ldots, \xb_{F_r}$, $e\in \mathrm{Simp}(\C)$ and $F$ is a circuit of $\C$ which contains $e$, then the ideal $I\left(\bar{\mathcal{C}}\right)+ \left( \xb_F \right)$ has also linear quotients with respect to $\xb_{F_1}, \ldots, \xb_{F_r}, \xb_F$.
\end{itemize}
\end{lem}

\begin{cor} \label{simplicial subclutter has linear quotients-cor}
Let $\mathcal{C}$ be a $d$-uniform clutter such that the ideal $I\left(\bar{\mathcal{C}}\right)$ has linear quotients. If $\mathcal{D}$ is a simplicial subclutter of $\mathcal{C}$, then the ideal $I\left(\bar{\mathcal{D}}\right)$ has linear quotients and hence linear resolution.
\end{cor}

\begin{proof}
We use induction on $|\bar{\mathcal{D}}|$ to prove the assertion. Since $\bar{\mathcal{C}} \subseteq \bar{\mathcal{D}}$, the induction base lies on $\mathcal{D}=\mathcal{C}$. In this case, the result follows from our hypothesis. Assume that the results holds for all simplicial subclutter $\mathcal{D}'$ of $\mathcal{C}$ with $|\bar{\mathcal{D}'}|< |\bar{\mathcal{D}}|$. Now by Definition~\ref{Definition of simplicial subclutter} there exists a $d$-uniform clutter $\mathcal{D}'$, $e \in \mathrm{Simp} \left( \mathcal{D}' \right)$ and $F \in \mathcal{D}'$ containing $e$, such that
\begin{itemize}
\item[(i)] $\mathcal{D}'$ is a simplicial subclutter of $\mathcal{C}$, and
\item[(ii)] $\mathcal{D} = \mathcal{D}' \setminus \{F\}$.
\end{itemize}
Then$|\bar{\mathcal{D}'}|< |\bar{\mathcal{D}}|$ and so induction hypothesis  implies that the ideal $I \left( \bar{\mathcal{D}}' \right)$ has linear quotients. Note that $I \left( \bar{\mathcal{D}} \right) = I \left( \bar{\mathcal{D}}' \right) + \left( \textbf{x}_F \right)$. Lemma~\ref{simplicial subclutter has linear quotients}(b) now yields the desired conclusion.
\end{proof}

The following example shows that the converse of Corollary~\ref{simplicial subclutter has linear quotients-cor} is not true in general.

\begin{ex}
Let $\C$ be a $d$-uniform chordal clutter such that the ideal $I(\bar{\C})$ does not have linear quotients; such clutter exists, see for example \cite[Example 3.14]{BYZ}. Then $\varnothing$ is a simplicial subclutter of $\C$ and $I(\bar{\varnothing})= I(\C_{n,d})$ has linear quotients. Consequently,  the circuit ideal of the complement of a simplicial subclutter of a clutter $\C$ may have linear quotients, while the original circuit ideal $I(\bar{\C})$  may not.
\end{ex}

\subsection{Simplicial subclutters and subadditivity problem}
Let $S=\mathbb{K}[x_1, \ldots, x_n]$ be the polynomial ring over a field $\mathbb{K}$ endowed with standard grading (i.e. $\deg (x_i) =1$, for all $1 \leq i \leq n$). For a homogeneous ideal $I \subset S$, let 
\begin{align*}
t_i(I) & = \max \{ j \colon \quad \beta_{i,j}(S/I) \neq 0 \} \\
& = \max \{ j \colon \quad \mathrm{Tor}_{i}^{S} \left( \mathbb{K}, {S/I} \right)_j \neq 0 \}
\end{align*}
and $t_i(I)= - \infty$, if it happens that $\mathrm{Tor}_{i}^{S} \left( \mathbb{K}, {S/I} \right) = 0$. Also define
$$r_i(I) = t_i(I) -i.$$
%With this notation, we have:
%$$\mathrm{reg} (S/I) = \max \{t_i(I)-i \colon i \in \mathbb{Z} \}.$$

A homogeneous ideal $I \subset S$ is said to satisfy the \textit{subadditivity} condition, if 
\begin{equation} \label{Subadditivity problem 2}
t_{i+j} \left( {I} \right) \leq t_i \left( {I} \right) + t_j \left( {I} \right),
\end{equation}
for all $i,j$ with $i+j \leq \mathrm{pd}(S/I)$.

Subadditivity problem has been studied in \cite{Nevo, Bigdeli-Herzog, Conca, ECU, Faridi, Herzog-Srinivasan, Khoury-Sirinivasan, McCullough, YazdanPour}. No counterexample is known for the validity of (\ref{Subadditivity problem 2}) for monomial ideals, while for arbitrary homogeneous ideals, (\ref{Subadditivity problem 2}) is not true in general (c.f. \cite[Sec. 6]{Conca}).

\medskip
Let $\mathcal{D}$ be a simplicial subclutter of a $d$-uniform clutter $\mathcal{C}$. Let $I= I(\bar{\mathcal{C}})$ and $J= I(\bar{\mathcal{D}})$ be the corresponding associated ideals. We will show that $I$ satisfies the subadditivity condition, if and only if so does $J$. Indeed, we show a stronger result, which states that the Betti diagram of $I$ has a \textit{special} shape, if and only if so does $J$. Here \textit{special} shape means the definition as it is stated in \cite{Bigdeli-Herzog}. In the following, we recall the required prerequisite to reach this aim.

%The special shape of the Betti diagram is in a way that the sequence
%of numbers $r_0(I), r_1(I), \ldots, r_{\rho}(I)$ with $\rho = \mathrm{pd}(S/I)$ form a
%concave sequence until they reach the regularity and from this point on are non-increasing  \cite{Bigdeli-Herzog}. Indeed,

\begin{defn}[{\cite[Definition 1]{Bigdeli-Herzog}}] \label{definition of special shape}
Let $I$ be a graded ideal. The Betti diagram of $S/I$ with regularity $c$ is said to have a \textit{special} shape, if
\begin{itemize}
\item[\rm (i)] $r_0(I) \leq r_1(I)\leq \cdots \leq  r_{g}(I)$ where $g$ is the smallest integer such that $r_{g}(I)=c$.
\item[\rm (ii)] $r_{i+1}(I) \leq r_i (I)$, for $g \leq i \leq \mathrm{pd} (S/I)$.
\end{itemize}
\end{defn}

As it is mentioned in \cite{Bigdeli-Herzog}, not all Betti diagrams of monomial ideals have a special shape. This is not even the case for the edge ideal of a graph. A simplest such example is the edge ideal of the $5$-cycle whose Betti diagram violates condition (i) of Definition~\ref{definition of special shape}. However, it is shown in \cite[Lemma 1]{Bigdeli-Herzog} that any homogenous ideal whose Betti diagram has a special shape satisfies subadditivity.

\begin{prop} \label{same non-linear strands and subadditivity}
Let $0 \neq I \subseteq J$ be homogeneous ideals generated in degree $d$ such that
$$\beta_{i,j} (I)= \beta_{i,j}(J),$$
for all $i,j$ with $j-i>d$. Then
\begin{itemize}
\item[\rm (a)] $I$ satisfies the subadditivity condition if and only if $J$ does so.
\item[\rm (b)] $I$ has a special Betti diagram if and only if $J$ does so.
\end{itemize}
\end{prop}

\begin{proof}
First of all note that $\beta_{i, i+d} (I) \leq \beta_{i, i+d}(J)$, for all $i$ by \cite[Lemma 2.3]{BYZ}. Hence our assumption on the Betti numbers implies that $ \mathrm{pd}(S/I) \leq \mathrm{pd}(S/J)$, and
\begin{equation} \label{internal equation 15}
t_i(J) = \begin{cases}
t_i (I), & \text{if } 0 \leq i \leq \mathrm{pd}(S/I) \\
d+i-1, & \text{if } \mathrm{pd}(S/I)< i \leq \mathrm{pd}(S/J).
\end{cases}
\end{equation}

(a) Assume that the ideal $J$ satisfies the subadditivity condition. Then it follows from (\ref{internal equation 15}) that for all $i,j$ with $i+j \leq \mathrm{pd}(S/I)$, one has 
\begin{align*}
t_{i+j}(I) = t_{i+j}(J) \leq t_i(J) + t_j(J) = t_i(I) + t_j(I). 
\end{align*}
So that the ideal $I$ satisfies  the sabadditivity condition as well. For the converse, suppose that the ideal $I$ satisfies the subadditivity condition. If $i+j \leq \mathrm{pd}(S/I)$, then as above we have
\begin{align*}
t_{i+j}(J)  \leq  t_i(J) + t_j(J).
\end{align*}
 Assume that $\mathrm{pd}(S/I) < i+j \leq \mathrm{pd}(S/J)$. Then, 
\begin{align*}
t_{i+j}(J) =i+j+d-1 \leq (i+d-1) + (j+d-1) \leq t_i(J) + t_j(J).
\end{align*}
(b) A similar argument as in (a) yields the desired conclusion.
\end{proof}

Proposition~\ref{same non-linear strands and subadditivity} in combination with Corollary~\ref{chordallin}(a) leads to the following corollary.
\begin{cor}
Let $\C$ be a $d$-uniform clutter and $\mathcal{D}$ be a simplicial subclutter of $\C$. Then 
\begin{itemize}
\item[\rm (a)] the ideal $I(\bar{\C})$ satisfies subadditivity condition if and only if the ideal $I(\bar{\mathcal{D}})$ does so. 
\item[\rm (b)] the ideal $I(\bar{\C})$ has a special Betti diagram if and only if $I(\bar{\mathcal{D}})$ does so.
\end{itemize}
\end{cor}

\section{Simplicial  clutters} \label{complete}

In this section, we consider the simplicial subclutters of a complete clutter. A  simplicial subclutter of a complete clutter is called a {\em simplicial clutter}. Since $I(\overline{\C_{n,d}})=0$, one would expect that the Betti table of the ideals of the simplicial clutters can be explicitly determined through the data obtained from the simplicial sequences.  
The following result which is a direct conclusion of Corollaries~\ref{simplicial subclutter has linear quotients-cor} and \ref{graded linear strands-2}, gives rise to computing Betti numbers of simplicial subclutters of the complete clutters.

\begin{cor}\label{resolution of a simplicial subcluuter of complete clutter-cor}
	Let $\C$ be a simplicial clutter  obtained from $\C_{n,d}$ by the simplicial sequence $\textbf{e}= e_1, \ldots, e_r$ and the sets $A_1, \ldots, A_r$ as defined in Definition~\ref{Definition of simplicial subclutter}. Then the ideal $I=I(\bar{\C})$ has linear quotients and 
	\begin{equation} \label{resolution of a simplicial subcluuter of complete clutter}
	\beta_{i,i+d} (I) = \sum\limits_{k=1}^{r} \sum\limits_{j=0}^{|A_k|-1} {t_k+j \choose i},
	\end{equation}
	where $t_1=0$ and $t_k = |\{ F \in \mathop{\bigcup}\limits_{j=1}^{k-1} A_j \colon \; e_k \subset F \}|$, for $2 \leq k\leq r$. 
\end{cor}

According to Corollary~\ref{resolution of a simplicial subcluuter of complete clutter-cor}, % implies, in particular, that the ideal $I\left(\bar{\mathcal{C}}\right)$ has linear quotients if $\C$ is a simplicial subclutter of a complete clutter. A class of square-free monomial ideals with linear quotients is the class of square-free stable ideals. 
 a class of square-free monomial ideals with linear quotients is the class of ideals obtained from simplicial subclutters of the complete clutter.  
Another class of square-free monomial ideals with linear quotients is the class of square-free stable ideals (see \cite[Problem~8.8(b)]{HHBook}). 
A  square-free monomial ideal $I \subset S$ is called {\em square-free stable} if for all square-free monomials $u \in I$ and for all $j< m(u)$ such that $x_j$ does not divide $u$ one has $x_j(u/x_{m(u)})\in I$, where $m(u)=\max\{i:\ x_i\text{ divides } u\}$. Note that this exchange property  needs only be checked for the monomials in $\mathcal{G}(I)$ (see \cite[Problem~6.9]{HHBook}). In the following theorem we see that the class of equigenerated square-free stable ideas is contained in the class of ideals associated to simplicial subclutters of complete clutters.
 %class of ideals are associated to simplicial subclutters.
 
\begin{thm} \label{stable ideals are simplicial subclutters of complete clutters}
Let $I$ be a square-free stable ideal in $S$ generated in degree $d$, and let   $\mathcal{C}$ be a $d$-uniform clutter on the vertex set $[n]$ with $I=I(\bar{\mathcal{C}})$. Then $\mathcal{C}$ is a simplicial clutter. 
\end{thm}

\begin{proof}
Let $I=(\mathbf{x}_{F_1},\ldots,\mathbf{x}_{F_r})$ (i.e. $\bar{\mathcal{C}}=\{F_1,\ldots,F_r\}$). Suppose ${\bf x}_{F_1}<\cdots< {\bf x}_{F_r}$, where $<$ is the lexicographic  order induced by $x_1< \cdots < x_n$.  For $1\leq i\leq r$, let $$\mathcal{C}_i=\mathcal{C}\cup\{F_i,\ldots,F_r\}.$$ Note that $\C_1=\C_{n,d}$, $I_1:=I(\bar{\C}_1)=0$  and for $i>1$,  $$I_i:=I(\bar{\mathcal{C}_i})=(\mathbf{x}_{F_1},\ldots,\mathbf{x}_{F_{i-1}}).$$ It is seen that for $1\leq i\leq r$, setting $e_i=F_i\setminus \{m({\bf x}_{F_i})\}$ we have $N_{\C_i}[e_i]=e\cup\{m({\bf x}_{F_i}), m({\bf x}_{F_i})+1,\ldots,n\}$.
Let $F$ be a $d$-subset of $N_{\C_i}[e_i]$. We show that $F\in \C_i$. If $e\subset F$, then by definition of $N_{\C_i}[e_i]$ we have $F\in \C_i$. Suppose $e\not\subset F$,  $F=\{a_1,\ldots, a_d\}$ with $a_1<\cdots<a_d$. Since $m({\bf x}_{F_i})>k$ for all $k\in e$ we have $a_{d-1}, a_d\in \{m({\bf x}_{F_i}), m({\bf x}_{F_i})+1,\ldots,n\}$. If $F\notin \C_i$, then $\mathbf{x}_F\in I_i$ which implies that $\mathbf{x}_F< \mathbf{x}_{F_i}$. Hence $m(\mathbf{x}_F)=a_d\leq m(\mathbf{x}_{F_i})$. Since $a_d\in \{m({\bf x}_{F_i}), m({\bf x}_{F_i})+1,\ldots,n\}$ we have $m(\mathbf{x}_F)=a_d= m(\mathbf{x}_{F_i})$. Therefore $a_{d-1}< m(\mathbf{x}_{F_i})$ which is a contradiction.  Hence  $e_i$ is a simplicial maximal subcircuit of $\C_i$. It follows that  $\mathcal{C}=\mathcal{C}_{n,d}\setminus A_1\setminus \cdots\setminus A_r$ is a simlicial clutter, where $A_i=\{F_i\}$ with $e_i\subset F_i$. %and $e_1\in \rm{Simp}(\C_{n,d})$ and $e_i\in \mathrm{Simp}(\mathcal{C}_{n,d}\setminus A_1\setminus \cdots\setminus A_{i-1})$ for  $i>1$. 

\end{proof}
%As an application of the Corollary~\ref{resolution of a simplicial subcluuter of complete clutter-cor} and Theorem~\ref{stable ideals are simplicial subclutters of complete clutters} we may obtain the minimal free resolution of an equigenrated stable ideal.
%As it is seen in Theorem~\ref{stable ideals are simplicial subclutters of complete clutters}, if for a given $d$-uniform clutter $\mathcal{C}$ the ideal $I=I(\bar{\mathcal{C}})$ is a square-free stable ideal, then ${\mathcal{C}}$ is a simplicial subclutter of $\mathcal{C}_{n,d}$ for some $n,d$. 
Theorem~\ref{stable ideals are simplicial subclutters of complete clutters} allows us to use the formula given in Corollary~\ref{resolution of a simplicial subcluuter of complete clutter-cor} to compute the Betti numbers of $I$, in case that the ideal $I=I(\bar{\mathcal{C}})$ is a square-free stable ideal. Note that one may prove Theorem~\ref{stable ideals are simplicial subclutters of complete clutters} by induction on  $|\bar{\mathcal{C}}|$, (which is a shorter proof). But  we need the structure of the current proof in the following example.

\begin{ex}[Betti numbers of equigenerated square-free stable ideals]
Let %$\C$ be a $d$-uniform clutter on the vertex set $[n]$ such that the ideal 
$I=I(\bar{\C})=(\mathbf{x}_{F_1},\ldots,\mathbf{x}_{F_r})$ be a square-free stable ideal, where $\C$ is a $d$-uniform clutter on $[n]$.  Suppose ${\bf x}_{F_1}<\cdots< {\bf x}_{F_r}$, where $<$ is the lexicographic  order induced by $x_1< \cdots < x_n$.  %For $1\leq i\leq r$, let $\mathcal{C}_i=\mathcal{C}\cup\{F_i,\ldots,F_r\}$. Note that $\C_1=\C_{n,d}$, $I_1:=I(\bar{\C}_1)=0$  and for $i>1$,  $I_i:=I(\bar{\mathcal{C}_i})=(\mathbf{x}_{F_1},\ldots,\mathbf{x}_{F_{i-1}})$. It is seen that for $1\leq i\leq r$, setting $e_i=F_i\setminus \{m({\bf x}_{F_i})\}$ we have $N_{\C_i}[e_i]=e\cup\{m({\bf x}_{F_i}), m({\bf x}_{F_i})+1,\ldots,n\}$.
%Let $F$ be a $d$-subset of $N_{\C_i}[e_i]$. We show that $F\in \C_i$. If $e\subset F$, then by definition of $N_{\C_i}[e_i]$ we have $F\in \C_i$. Suppose $e\not\subset F$,  $F=\{a_1,\ldots, a_d\}$ with $a_1<\cdots<a_d$. Since $m({\bf x}_{F_i})>k$ for all $k\in e$ we have $a_{d-1}, a_d\in \{m({\bf x}_{F_i}), m({\bf x}_{F_i})+1,\ldots,n\}$. If $F\notin \C_i$, then $\mathbf{x}_F\in I_i$ which implies that $\mathbf{x}_F< \mathbf{x}_{F_i}$. Hence $m(\mathbf{x}_F)=a_d\leq m(\mathbf{x}_{F_i})$. Since $a_d\in \{m({\bf x}_{F_i}), m({\bf x}_{F_i})+1,\ldots,n\}$ we have $m(\mathbf{x}_F)=a_d= m(\mathbf{x}_{F_i})$. Therefore $a_{d-1}< m(\mathbf{x}_{F_i})$ which is a contradiction.  Hence  $e_i$ is a simplicial maximal subcircuit of $\C_i$. It follows that  $\mathcal{C}=\mathcal{C}_{n,d}\setminus A_1\setminus \cdots\setminus A_r$ is a simlicial subclutter of $\C_{n,d}$, where $A_i=\{F_i\}$ with $e_i\subset F_i$. %and $e_1\in \rm{Simp}(\C_{n,d})$ and $e_i\in \mathrm{Simp}(\mathcal{C}_{n,d}\setminus A_1\setminus \cdots\setminus A_{i-1})$ for  $i>1$. %A similar argument as in  the proof of Theorem~\ref{stable ideals are simplicial subclutters of complete clutters} shows that for $i>1$ the ideal $I_i$  is square-free stable. 
As seen in the proof of Theorem~\ref{stable ideals are simplicial subclutters of complete clutters} we have $\mathcal{C}=\mathcal{C}_{n,d}\setminus A_1\setminus \cdots\setminus A_r$, where $A_i=\{F_i\}$ with $F_i=e_i\cup\{m(\mathbf{x}_{F_i})\}$ and $e_i\in\mathrm{Simp}(\mathcal{C}_{n,d}\setminus A_1\setminus \cdots\setminus A_{i-1})$.
 
Now we compute $t_k$ (for $k>1$), where $t_k$ is  defined  in Corollary~\ref{resolution of a simplicial subcluuter of complete clutter-cor}. Indeed, we should count the number of  $F$ in $\{F_1,\ldots,F_{k-1}\}$ for which $e_k\subset F$. Since  $I$ is square-free stable  we have $x_i\mathbf{x}_{F_k}/x_{m(\mathbf{x}_{F_k})}\in I$ for all $i<m(\mathbf{x}_{F_k})$ such that $x_i$ does not divide $\mathbf{x}_{F_k}$. But  $x_i\mathbf{x}_{F_k}/x_{m(\mathbf{x}_{F_k})}< \mathbf{x}_{F_k}$. Thus  $\{i\}\cup e_k\in \{F_1,\ldots, F_{k-1}\}$ for all $i<m(\mathbf{x}_{F_k})$ with $i\notin e_k$. So the number of $F=\{i\}\cup e_{k}$ with $i<m(\mathbf{x}_{F_k})$ and $i\notin e_k$ in  $\{F_1,\ldots, F_{k-1}\}$ is  $(m(\mathbf{x}_{F_k})-1)-(d-1)$. Note that  $\{i\}\cup e_k\notin \{F_1,\ldots, F_{k-1}\}$ for $i>m(\mathbf{x}_{F_k})$, because ${\bf x}_{F_k}< x_i{\bf x}_{e_k}$ for $i>m(\mathbf{x}_{F_k})$. Consequently, $t_k=m(\mathbf{x}_{F_k})-d$ for $k>1$. It is  also clear that $m({\bf x}_{F_1})=d$. Thus $t_1=0= m({\bf x}_{F_1})-d$. 
Therefore we get the following
\begin{align*}
\beta_{i, i+d} (I) &=  \sum\limits_{k=1}^{r} \sum\limits_{j=0}^{|A_k|-1} {t_k+j \choose i}\tag{\text{using Corollary~\ref{resolution of a simplicial subcluuter of complete clutter-cor}}}\\
&=\sum\limits_{k=1}^{r} {m(\mathbf{x}_{F_k})-d \choose i}\\
&=\sum\limits_{u\in \mathcal{G}(I)} {m(u)-d \choose i}
\end{align*}
The last formula matches the formula given in \cite[Corollary~7.4.2]{HHBook} for square-free stable ideals.
\end{ex}

A sequence $\bm{\beta} =(\beta_0,\ldots,\beta_{n-1})$ of  integers is called the {\em Betti sequence} of a graded ideal if there exists  a graded ideal $I\subset S=\mathbb{K}[x_1,\ldots,x_n]$  such that $\beta_i=\beta_i(I)$, for all $i$. In \cite[Theorem~3.3]{MJAR}, it is proved that any Betti sequence of a graded ideal with linear resolution  is the Betti sequence of the ideal associated to a chordal clutter. Since the ideals of simplicial clutters have linear resolution too, the question arises whether all ideals which have a linear resolution occur as the circuit ideal of simplicial clutters.  In the following result we give an affirmative answer to this question.
\begin{thm} \label{all}
Let $\bm{\beta}=(\beta_0,\ldots,\beta_{n-1})$ be a sequence of integers. Then $\bm{\beta}$ is a Betti sequence of a graded  ideal with linear resolution if and only if $\bm{\beta}$ is the Betti sequence of the  circuit ideal of the complement of a simplicial clutter.
\end{thm}

\begin{proof}
Let $I$ be a graded ideal in $S$ with $d$-linear resolution. As it is stated in the proof of \cite[Theorem 3.3]{MJAR}, there exists a square-free (strongly) stable ideal, say $J$, such that the Betti sequence of $I$ coincides with the Betti sequence of $J$. The assertion now follows from Theorem~\ref{stable ideals are simplicial subclutters of complete clutters}.
\end{proof}
Theorem~\ref{all} is quite surprising because it confirms that the Betti sequence of any graded ideal with linear resolution over all fields is the Betti sequence of an ideal with linear quotients. Note that not all ideals with linear resolution over all fields have linear quotients (e.g \cite[Example~3.14]{BYZ}).

\section{Contractible complexes vs. Chordal clutters}\label{44}
In Corollary~\ref{resolution of a simplicial subcluuter of complete clutter-cor}, it is stated that $I(\bar{\C})$ has linear quotients if $\C$ is a simplicial clutter. Simple examples show that the class of simplicial clutters are not equivalent to the class of equigenerated square-free monomial ideals with linear quotients. It follows that they are not equivalent to the class of square-free ideals with linear resolution over all fields.
\begin{ex}\label{counter2}
Let $G$ be the graph in Figure~\ref{parvaneh}.
\begin{figure}[H]
\centering
\begin{tikzpicture}[line cap=round,line join=round,>=triangle 45,x=0.5cm,y=0.5cm]
\clip(2.,1.3) rectangle (12.,5.5);
\draw [line width=0.8pt] (5.,5.)-- (3.,2.);
\draw [line width=0.8pt] (3.,2.)-- (7.,2.);
\draw [line width=0.8pt] (7.,2.)-- (5.,5.);
\draw [line width=0.8pt] (7.,2.)-- (9.,5.);
\draw [line width=0.8pt] (9.,5.)-- (11.,2.);
\draw [line width=0.8pt] (11.,2.)-- (7.,2.);
\draw (2.7,2.0) node[anchor=north west] {\begin{scriptsize}$1$\end{scriptsize}};
\draw (4.84,5.7) node[anchor=north west] {\begin{scriptsize}$2$\end{scriptsize}};
\draw (6.75,2) node[anchor=north west] {\begin{scriptsize}$3$\end{scriptsize}};
\draw (8.8,5.7) node[anchor=north west] {\begin{scriptsize}$4$\end{scriptsize}};
\draw (10.8,2) node[anchor=north west] {\begin{scriptsize}$5$\end{scriptsize}};
\begin{scriptsize}
\draw [fill=black] (5.,5.) circle (1.5pt);
\draw [fill=black] (3.,2.) circle (1.5pt);
\draw [fill=black] (7.,2.) circle (1.5pt);
\draw [fill=black] (9.,5.) circle (1.5pt);
\draw [fill=black] (11.,2.) circle (1.5pt);
\end{scriptsize}
\end{tikzpicture}
\caption{graph G}\label{parvaneh}
\end{figure}
We have $$I(\bar{G})=(x_1x_2,x_1x_3,x_2x_3,x_1x_2,x_3x_4,x_3x_5,x_4x_5)$$ which has linear quotients, and hence linear resolution over all fields. But it is easy to check that $G$ is not a simplicial clutter. 
\end{ex}

Now let us consider the dual class of simplicial clutters: the class of clutters with $\varnothing$ as their simplicial subclutter. We know that the ideal attached to the elements of this class have linear resolution over all fields.  Moreover, the class of chordal clutters is a  subclass of this class. So far, we do not know if the class of chordal clutters is indeed equal to the class of clutters with $\varnothing$ as their simplicial subclutter. However, the main idea for defining the class of chordal clutters is to find the largest possible class of uniform clutters such that the associated circuit ideal has a linear resolution over all fields. Note that the class of chordal clutters is not a subclass of the class of simplicial clutters, for instance, the graph $G$ in Example~\ref{counter2}  is chordal, while it is not a simplicial clutter.

In \cite{BYZ} it is proved that the class $\mathfrak{C}_d$ of chordal clutters strictly contains several previously known classes of  clutters, whose associated ideals have linear resolution over all fields, such as  classes defined by Emtander \cite{Emtander} and Woodroofe \cite{Woodroofe}.

 It is  shown in \cite[Theorem~2.5]{NZ} that if the circuit ideal of a uniform clutter $\mathcal{C}$ is square-free stable, then $\bar{\mathcal{C}}$ is chordal.  Since square-free lexsegment ideals and square-free strongly stable ideals are square-free stable, it follows that these classes also come form chordal clutters. 
 
 Another subclass of $\mathfrak{C}_d$ is associated to vertex decomposable simplicial complexes: Let $\Delta$ be the clique complex of a $d$-uniform clutter $\mathcal{C}$ with the property that its Alexander dual, $\Delta^\vee$,  is vertex decomposable. Then $\mathcal{C}$ is chordal, \cite[Theorem~3.5]{Nik}. This, in particular, implies that the matroidal ideals also come from chordal clutters. Indeed, it follows from the fact that any matroidal ideal is the Stanley-Reisner ideal of the dual of a vertex decomposable simplicial complex, \cite[Proposition~7]{ER}. 
 
 These evidence seemed to strengthen the guess that all square-free monomial ideals with linear resolution over all fields are associated to chordal clutters. In support of this guess,  in \cite{BYZ} the authors asked whether there exists a $d$-uniform clutter $\mathcal{C}$ such that the ideal $I(\bar{\mathcal{C}})$ has a linear resolution over all fields, while $\mathcal{C}$ does not belong to the class $\mathfrak{C}_d$,   \cite[Question~1]{BYZ}. Recently, due to a discussion with Eric Babson, it turned out that the answer  is positive. In the following we first find a class of non-chordal clutters whose ideals have linear resolution over all fields. Then we give explicit examples of this class.

\medskip
Before finding the mentioned class, we  first go through the following Lemma.  Let $I \subset S$ be a square-free monomial ideal. For each $j$, we write
$I_{[j]}$ for the ideal generated by all the square-free monomials of degree $j$ belonging
to $I$.

\begin{lem} \label{regularity via component}
Let $\Delta$ be a simplicial complex on the vertex set $[n]$, $I=I_\Delta \subset S$ its Stanley-Reisner ideal and $d = \max \{\deg (u) \colon \quad u \in \mathcal{G}(I) \}$. For $t \geq d$, let 
\begin{equation*}
\Delta_t  = \Delta \cup \langle [n] \rangle^{[t-2]}.
\end{equation*}
Then
 $I_{\Delta_t} = I_{[t]}$ and 
 if $\textbf{\rm \textbf{a}} \in \mathbb{Z}^n$ and $|\textbf{\rm \textbf{a}}| >i+t$, then $\beta_{i,\textbf{\rm \textbf{a}}}(I) = \beta_{i,\textbf{\rm \textbf{a}}}(I_{[t]})$. In particular, $\beta_{i,j}(I) = \beta_{i,j}(I_{[t]})$, if $j-i>t$.
\end{lem}

\begin{proof}
If $\mathrm{\textbf{x}}_F \in I_{[t]}$, then $|F| \geq t$ and there exists a subset $G \subseteq F$ with $\mathrm{\textbf{x}}_G \in \mathcal{G}(I)$ and $\mathrm{\textbf{x}}_G |\mathrm{\textbf{x}}_F$. Since $\textbf{x}_G \in \mathcal{G}(I)$ and $G\subseteq F$, we conclude that $F \notin \Delta$. Since $|F|\geq t$ we have  $F \notin \Delta_t$. Hence $\textbf{x}_F \in I_{\Delta_t}$.

Conversely, take an element $\textbf{x}_F \in I_{\Delta_t}$. Then, $F \notin \Delta$ and $|F| \geq t$.  Hence $\mathbf{x}_F\in I_\Delta$ with $\deg (\mathbf{x}_F)\geq t$. This implies that $\textbf{x}_F \in I_{[t]}$.

 Let $t \geq d$. It follows from definition of $\Delta_t$ that $\Delta_t^{[i]} = \Delta^{[i]}$, for all $i >t-2$. Hence $\widetilde{H}_i \left(\left(\Delta_t\right)_W; K \right) = \widetilde{H}_i \left(\Delta_W; K \right)$, for all $i>t-2$ and $W \subset [n]$. Now, let $\textbf{\rm \textbf{a}} \in \mathbb{Z}^n$ with $|\textbf{\rm \textbf{a}}| >i+t$ and let $W= \mathrm{supp}(\textbf{\rm \textbf{a}})$. Then, by a theorem of Hochster, we have
\begin{align*}
\beta_{i,\textbf{\rm \textbf{a}}}(I_\Delta) &= \dim_\mathbb{K} \widetilde{H}_{|W|-i-2} \left( \Delta_W; \mathbb{K} \right)\\
&= \dim_K \widetilde{H}_{|W|-i-2} \left( \left(\Delta_t \right)_W; \mathbb{K}\right) = \beta_{i,\textbf{\rm \textbf{a}}}(I_{\Delta_t}).
\end{align*}
\end{proof}

\begin{rem}
Let $I$ be a square-free monomial ideal. It is easy to check that as a consequence of Lemma~\ref{regularity via component} we have 
\begin{align}
\reg \left(I_{[t]}\right) = \max \left\{ t, \; \reg(I)\right\}, \text{ for all $t\geq d$}.
\end{align}
Moreover, 
\begin{align}\label{algorithm}
\reg(I) = \min \left\{ t\colon \quad t \geq d \text{ and } I_{[t]} \text{ has a $t$-linear resolution} \right\}.
\end{align}
The formula (\ref{algorithm}) implies that one will get an algorithm for computing the regularity of an arbitrary  monomial ideal once there is an algorithm which decides whether a monomial ideal has a linear resolution.  
This formula  implies, in particular,  that if $I$ has a $d$-linear resolution, then $I_{[t]}$ has a $t$-linear resolution for all $t \geq d$.
\end{rem}

\begin{lem}[{proposed by a referee of \cite{BYZ}}] \label{chordal clutters collapse}
If $\C$ is a $d$-uniform chordal clutter on the vertex set $[n]$, then $\Delta(\C)$ collapses to a subcomplex of $\langle [n] \rangle^{[d-2]}$.
\end{lem}

\begin{proof}

Let $\textbf{\rm \textbf{e}}=e_1, \ldots, e_t$ be a simplicial order over $\C$ and $\Delta_0 := \Delta(\C)$. For all $i>0$, define
\begin{align*}
\Delta_i = \Delta_{i-1} \setminus \{ F \in \Delta_{i-1} \colon \, e_i \subseteq F\}.
\end{align*}
Note that  $N_{\C}[e_1]$ is the only facet of $\Delta_0$ which properly contains $e_1$ and for $i>1$, $N_{\C  \setminus e_1 \setminus \cdots \setminus e_{i-1} }[e_i]$ is the only facet of $\Delta_{i-1}$ which properly contains $e_i$. Therefore $\Delta_0$ collapses to $\Delta_t$. Since $ \Delta(\C \setminus e_1 \setminus \cdots \setminus e_{i})=\Delta_i\cup\langle e_1,\ldots, e_i\rangle$, it follows  that $\Delta_t =\langle [n] \rangle^{[d-2]} \setminus \{e_1, \ldots, e_t\}$.
%Then $N_{\C}[e_1]$ is the only facet of $\Delta := \Delta(\C)$ which properly contains $e_1$. Hence $\Delta$ collapses onto $\Delta_1 :=\Delta \setminus \{ F \in \Delta \colon \, e_1 \subseteq F\}$. Let $\Delta'_1 = \Delta(\C \setminus e_1)$. One should observe that $\Delta'_1 = \Delta_1 \cup \langle e_1 \rangle$ and $e_1$ is a facet in $\Delta'_1$. Now, $N_{\C \setminus e_1}[e_2]$ is the only facet of $\Delta'_1$ and so $\Delta_1$ which properly contains $e_2$. So that $\Delta_1$ collapses onto $\Delta_2 :=\Delta_1 \setminus \{ F \in \Delta_1 \colon \, e_2 \subseteq F\}$. Let $\Delta'_2 = \Delta(\C \setminus e_1 \setminus e_2)$. One should observe that $\Delta'_2 = \Delta_2 \cup \langle e_2 \rangle$ and $e_2$ is a facet in $\Delta'_2$.  Continuing this algorithm, let 
%\begin{align*}
%\Delta_i = \Delta'_{i-1} \setminus \{ F \in \Delta_{i-1} \colon \, e_i \subseteq F\},  \text{ and }
%\Delta'_i = \Delta(\C \setminus e_1 \setminus \cdots \setminus e_{i-1}).
%\end{align*}
%Then $\Delta_i = \Delta'_i \cup \langle e_i \rangle$ and $e_i$ is a facet of $\Delta_i$. Hence Thus $\Delta'_i$ collapses to a subset of $\langle [n] \rangle^{[d-2]}$, if $\Delta_i$ does so. Since $\C \setminus e_1 \setminus \cdots \setminus e_{r} = \varnothing$, we conclude that $\Delta_{r+1} = \Delta(\varnothing) = \langle [n] \rangle^{[d-2]}$. The above argument shows that $\Delta(\C)$ collapses to a subset of $\langle [n] \rangle^{[d-2]}$.
\end{proof}

In the following we find a class of clutters such that their associated ideal have a linear resolution over all fields, while the clutter is not chordal. In order to do this, first we recall some definitions from algebraic topology.

A face of a simplicial complex $\Delta$ is called a {\em free face} if it is properly contained in a unique facet of $\Delta$. By a  {\em simple collapse} of $\Delta$ we mean  the simplicial complex obtained from $\Delta$  by removing of all faces of $\Delta$ containing a free face. A simplicial complex $\Delta$ is said to \textit{collapsible} to a simplicial complex $\Delta'$, if $\Delta'$ is obtained from $\Delta$ by a sequence of simple collapses.
If a simplicial complex $\Delta$ collapses to the empty simplicial complex then $\Delta$ is called {\em collapsible}. Note that any collapsible simplicial complex is contractible. Recall that a simplicial complex is called {\em contractible} if its geometric realization, as a topological space, is contractible, i.e. if the identity map on it is homotopic to some constant map.

For a pure $(d-1)$-dimensional simplicial complex $\Delta$, let $\mathcal{C}_{\Delta}$ denote the $d$-uniform clutter whose circuits are the facets of $\Delta$. 

\begin{prop}\label{a class of counters}
Let $\Delta$ be a pure simplicial complex of dimension $d-1$ on the vertex set $[n]$ such that
\begin{itemize}
\item[\rm (i)] $\Delta$ is contractible;
\item[\rm (ii)] $\Delta$ does not have a free face;
\item[\rm (iii)] The clique complex of $\mathcal{C}_{\Delta}$ is of dimension $d-1$. 
\end{itemize} 
Then
\begin{itemize}
\item[\rm (a)] $\mathcal{C}_{\Delta}$ is not a chordal clutter.
\item[\rm (b)] The ideal $I (\bar{\C}_\Delta)$ has a linear resolution over all  fields.
\end{itemize}
\end{prop}

\begin{proof}
Let $\Delta'$ be the clique complex of $\C_{\Delta}$. Then since $\Delta$ and $\Delta'$ are both of dimension $d-1$ we have
$$\Delta' = \Delta \cup \langle [n] \rangle^{[d-2]}.$$

(a) Since $\Delta$ does not have a free face, 
%it does  not collapse to a subcomplex of $\langle [n] \rangle^{[d-2]}$. 
it follows that $\Delta'$ does not collapse to a subcomplex of  $\langle [n] \rangle^{[d-2]}$. So that $\C_{\Delta}$ is not chordal by an application of Lemma~\ref{chordal clutters collapse}.

(b) Since $\Delta$ is contractible, $\widetilde{H}_i \left( \Delta; \mathbb{K} \right) = 0$ for all $i$. Since $\dim \Delta = d-1$, we conclude that $\widetilde{H}_i \left( \Delta_W; \mathbb{K} \right) = 0$ for all $i>d-2$  and for all $W \subset [n]$. Thus
\begin{equation*}
\beta_{i, j} (I_\Delta) = \sum\limits_{\textbf{\rm \textbf{a}} \in \mathbb{Z}^n \atop |\textbf{\rm \textbf{a}}|=j} \beta_{i, \textbf{\rm \textbf{a}}} (I_\Delta) = \sum\limits_{W \subset [n] \atop |W|=j} \dim_{\mathbb{K}} \widetilde{H}_{j-i-2} \left(\Delta_W; \mathbb{K} \right) = 0,
\end{equation*}
for all $i$ and $j$ with $j-i>d$. Now since $\Delta= \Delta_d,$ as in Lemma~\ref{regularity via component}, it follows from Lemma~\ref{regularity via component} that $\beta_{i,j} \left( I (\bar{\C}_\Delta) \right) = \beta_{i,j} (I_{\Delta'})=0$, for all $i,j$ with $j-i>d$. This completes the proof.
\end{proof}
The following two examples are samples of the class defined in Proposition~\ref{a class of counters}. To the knowledge of the authors, there are a few known such examples. 
\begin{ex}[{A couterexample proposed by Eric Babson}]\label{duncehat example}
%Let $\mathcal{C}$ be the  $3$-uniform clutter in Figure~\ref{duncehatce} which is a triangulation of the dunce hat.
Let $\Delta$ be a pure $2$-dimensional simplicial complex whose geometric realization is a triangulation of a dunce hat, see Figure~\ref{duncehatce}. 
\begin{figure}[ht!]
		\begin{center}
		\begin{tikzpicture}[line cap=round,line join=round,>=triangle 45, scale=0.5]
				\definecolor{uuuuuu}{rgb}{0.28,0.28,0.27}
		\definecolor{eqeqeq}{rgb}{0.88,0.88,0.88}
		\fill[fill=black,fill opacity=0.1] (8.,6.) -- (3.,-2.) -- (13.,-2.) -- cycle;
		\fill[color=eqeqeq,fill=eqeqeq,fill opacity=0.1] (8.,2.) -- (7.,1.) -- (7.46,0.) -- (8.68,0.) -- (9.,1.) -- cycle;
		\draw (8.,6.)-- (3.,-2.);
		\draw (3.,-2.)-- (13.,-2.);
		\draw (13.,-2.)-- (8.,6.);
		\draw (8.,2.)-- (7.,1.);
		\draw (7.,1.)-- (7.46,0.);
		\draw (7.46,0.)-- (8.68,0.);
		\draw (8.68,0.)-- (9.,1.);
		\draw (9.,1.)-- (8.,2.);
		\draw (8.,2.)-- (8.,6.);
		\draw (7.,1.)-- (8.,6.);
		\draw (7.,1.)-- (6.09,2.9);
		\draw (7.,1.)-- (4.8,0.97);
		\draw (7.46,0.)-- (3.,-2.);
		\draw (4.85,0.97)-- (7.46,0.);
		\draw (7.46,0.)-- (6.48,-2.);
		\draw (8.68,0.)-- (6.48,-2.);
		\draw (8.,2.)-- (7.46,0.);
		\draw (8.,2.)-- (8.68,0.);
		\draw (8.68,0.)-- (10.,-2.);
		\draw (8.68,0.)-- (13.,-2.);
		\draw (9.,1.)-- (13.,-2.);
		\draw (9.,1.)-- (11.1,0.98);
		\draw (9.,1.)-- (9.87,3.0);
		\draw (8.,2.)-- (9.87,3.0);
		\begin{scriptsize}
		\draw [fill](8.,6.) circle (1.5pt);
		\draw (8.0,6.3) node {$1$};
		\draw [fill] (3.,-2.) circle (1.5pt);
		\draw (3.0,-2.35) node {$1$};
		\draw [fill] (13.,-2.) circle (1.5pt);
		\draw (13.0,-2.35) node {$1$};
		\draw [fill] (6.08,2.9) circle (1.5pt);
		\draw (5.75,3.0) node {$3$};
		\draw [fill] (4.85,0.97) circle (1.5pt);
		\draw(4.55,1.07) node {$2$};
		\draw [fill] (11.14,0.97) circle (1.5pt);
		\draw(11.5,1.07) node {$2$};
		\draw [fill] (9.87,3.0) circle (1.5pt);
		\draw(10.2,3.1) node {$3$};
		\draw [fill] (6.48,-2.) circle (1.5pt);
		\draw(6.48,-2.4) node {$3$};
		\draw [fill] (10.,-2.) circle (1.5pt);
		\draw (10.0,-2.4) node {$2$};
		\draw [fill] (8.,2.) circle (1.5pt);
		\draw (7.73,2.15) node {$6$};
		\draw [fill] (7.,1.) circle (1.5pt);
		\draw(6.6,1.23) node {$5$};
		\draw [fill] (7.46,0.) circle (1.5pt);
		\draw(7.5,-0.4) node {$4$};
		\draw [fill] (8.68,0.) circle (1.5pt);
		\draw (8.65,-0.45) node {$8$};
		\draw [fill] (9.,1.) circle (1.5pt);
		\draw (9.4,1.23) node {$7$};
		\end{scriptsize}
		\end{tikzpicture}
		\label{duncehatce}
	\end{center}
		\caption{A triangulation of the dunce hat}
	\end{figure}
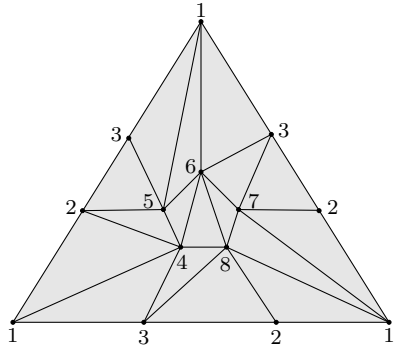
%It can be easily checked that the clutter $\mathcal{C}$ does not have any simplicial maximal subcircuit. Hence it is not chordal. But i
It is well-known that $\Delta$  is contractible \cite{Zeeman}. Moreover, it does not have a free face and, $\Delta(\C_\Delta)$ is $2$-dimensional simplicial complex. Thus $\Delta$ satisfies the assumptions of  Proposition~\ref{a class of counters}. Hence $I(\bar{\C}_{\Delta})$ has linear resolution over all fields, while $\C_\Delta$ is not chordal.
%; i.e.  it has  the homotopy type of a point. It follows that all the homotopy groups of  $\mathcal{C}$ are trivial. 
%Since $\Delta$ is a simplicial complex of dimension $2$, it follows that $\widetilde{H}_j (\Delta_W; \mathbb{K}) =0$, for all $j \geq 2$ and all $W \subset [8]$. 
%Let $\Delta'=\Delta\cup \langle [8]\rangle^{[1]}$. Then it follows from the  Mayer-Vietoris long exact sequence that $\widetilde{H}_j (\Delta'_W; \mathbb{K}) =0$, for all $j \geq 2$ and all $W \subset [8]$. Hence the ideal $I_{\Delta'}= I(\bar{\mathcal{C}})$ has a $3$-linear resolution over all fields. 
% the ideal $I(\bar{\mathcal{C}})$ has a $3$-linear resolution over all fields.
\end{ex}

\begin{ex}
Here we discuss Bing's celebrated example: House with two rooms \cite{Bi}, see Figure~\ref{bing}. 
\begin{figure}[H]
\includegraphics[width=4cm,height=3.5cm]{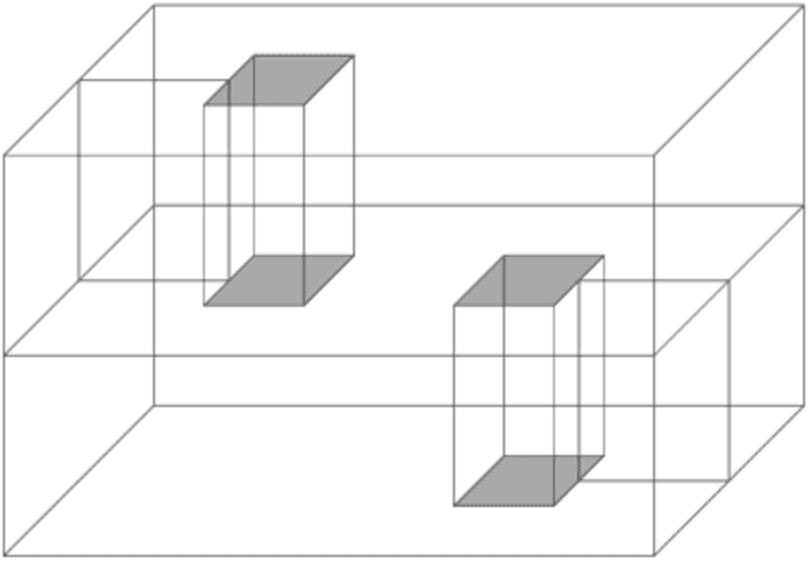}
\caption{Bing's House with two rooms}\label{bing}
\end{figure}

Consider a $2$-dimensional triangulation of this structure. Here we consider a triangulation explained in \cite[\S 5]{ADGL}, on $12$ vertices:
\begin{align*}
\mathcal{B}=
\Big\{&\{2,4,5\},\{ 2,3,5\},\{ 3,5,6\},\{ 3,4,6\},\{ 1,2,7\},\{ 2,7,8\},\{ 2,3,8\},\{ 3,8,9\},\{ 3,7,9\},\\
&\{4,7,8\},\{ 4,5,8\},\{ 5,6,8\},\{ 6,8,9\},\{ 6,7,9\},\{ 4,6,7\},\{ 1,4,7\},\{ 3,4,7\},\{ 1,2,3\},\\
&\{ 2,3,6\},\{ 2,5,6\},\{ 1,3,10\},\{ 3,10,11\},\{ 2,3,11\},\{ 2,11,12\},\{ 2,10,12\},\{ 4,10,11\},\\
&\{ 4,6,11\},\{ 5,6,11\},\{ 5,11,12\},\{ 5,10,12\},\{ 4,5,10\},\{ 1,4,10\},\{ 2,4,10\}
\Big\}.
\end{align*}
Let $\Delta$ be the simplicial complex with $\mathcal{F}(\Delta)=\mathcal{B}$. It is well-known that $\Delta$ is contractible; see e.g. \cite[Chapter 0]{HA}. The same discussion as in Example~\ref{duncehat example} shows that   
% Since $\Delta$ is a simplicial complex of dimension $2$, it follows that $\widetilde{H}_j (\Delta_W; \mathbb{K}) =0$, for all $j \geq 2$ and all $W \subset [12]$. Now let $\Delta'= \Delta \cup \langle [12] \rangle^{[1]}$. Then it follows from Mayer-Vietoris long exact sequence that $\widetilde{H}_j (\Delta'_W; \mathbb{K}) =0$, for all $j \geq 2$ and all $W \subset [12]$. Hence 
the ideal $I(\bar{\C}_\Delta)$ has a $3$-linear resolution over all fields, but  $\mathcal{\C}_\Delta$ is not chordal.
\end{ex}

\end{document}